\pdfoutput=1
\documentclass[12pt,letter]{amsart}
\usepackage{xcolor}
\usepackage[all,color]{xy}
\usepackage{amsmath,amssymb,amsthm,amscd,amsxtra,amsfonts,mathrsfs,graphicx,enumerate,bm,slashed,array,setspace,amsfonts} 
\input xy
\xyoption{all}
\newtheorem{Theorem}{Theorem}[section]
\newtheorem{Lemma}[Theorem]{Lemma}
\newtheorem{Proposition}[Theorem]{Proposition}
\newtheorem{Corollary}[Theorem]{Corollary}
\newtheorem{Example}[Theorem]{Example}
\newtheorem{Remark}[Theorem]{Remark}
\newtheorem{Definition}[Theorem]{Definition}

\newtheorem{Notation}[Theorem]{Notation}

\newtheorem*{Theorem A}{Theorem A}

\newcommand*{\overbar}[1]{\mkern 1.5mu\overline{\mkern-1.5mu#1\mkern-1.5mu}\mkern 1.5mu}
\advance\evensidemargin-.5in
\advance\oddsidemargin-.5in
\advance\textwidth1in

\begin{document}
\author{Charlie Beil}
 \address{Institut f\"ur Mathematik und Wissenschaftliches Rechnen, Universit\"at Graz, Heinrichstrasse 36, 8010 Graz, Austria.}
 \email{charles.beil@uni-graz.at}
 \title[Dimer algebras, ghor algebras, and cyclic contractions]{Dimer algebras, ghor algebras,\\ and cyclic contractions}
 \date{}

\begin{abstract}
A ghor algebra is the path algebra of a dimer quiver on a surface, modulo relations that come from the perfect matchings of its quiver.
Such algebras arise from abelian quiver gauge theories in physics.
We show that a ghor algebra $\Lambda$ on a torus is a dimer algebra (a quiver with potential) if and only if it is noetherian, and otherwise $\Lambda$ is the quotient of a dimer algebra by homotopy relations.
Furthermore, we classify the simple $\Lambda$-modules of maximal dimension and give an explicit description of the center of $\Lambda$ using a special subset of perfect matchings. 
In our proofs we introduce formalized notions of Higgsing and the mesonic chiral ring from quiver gauge theory. 
\end{abstract}

\maketitle

\section{Introduction}

Let $k$ be an algebraically closed field.
A dimer quiver $Q$ is a quiver that embeds in a compact surface $\Sigma$ such that each connected component of $\Sigma \setminus Q$ is simply connected and bounded by an oriented cycle, called a unit cycle.
A perfect matching $x$ of $Q$ is a subset of arrows such that each unit cycle contains precisely one arrow in $x$.
We will often assume that $Q$ is nondegenerate, that is, each arrow is contained in at least one perfect matching.
One can imagine associating a different color to each perfect matching, and coloring each arrow $a$ by all the perfect matchings that contain $a$.  In this article we introduce the \textit{ghor algebra} of $Q$, which is a quotient of the path algebra $kQ$ where paths with coincident heads and coincident tails are deemed equal if they have the same coloring.

More precisely, let $\mathcal{P}$ be the set of perfect matchings of $Q$, and let $n := |Q_0|$ be the number of vertices of $Q$.
Consider the algebra homomorphism from $kQ$ to the $n \times n$ matrix ring over the polynomial ring generated by $\mathcal{P}$,
\begin{equation*} \label{eta}
\eta: kQ \to M_n\left(k[\mathcal{P}] \right),
\end{equation*}
defined on the vertices $e_i \in Q_0$ and arrows $a \in Q_1$ by
\begin{equation} \label{eta def}
e_i \mapsto e_{ii} \ \ \ \text{ and } \ \ \ a \mapsto e_{\operatorname{h}(a),\operatorname{t}(a)} \prod_{\substack{x \in \mathcal{P}:\\ x \ni a}} x,
\end{equation}
and extended multiplicatively and $k$-linearly to $kQ$.
The ghor algebra of $Q$ is the quotient
$$\Lambda := kQ/\operatorname{ker} \eta.$$
$\Lambda$ may be viewed as a tiled matrix ring by identifying it with its induced image under $\eta$.
In companion articles we will find that ghor algebras, although often nonnoetherian, have rich geometric, algebraic, and homological properties \cite{B2,B3,B5,B6}.\footnote{`Ghor' is Klingon for surface.}

Ghor algebras generalize cancellative dimer algebras.
The \textit{dimer algebra} of a dimer quiver $Q$ is the superpotential algebra $A = kQ/I$, where $I$ is the ideal
\begin{equation} \label{I def}
I := \left\langle p-q \ | \ \exists a \in Q_1 \text{ s.t.\ } ap,aq \text{ are unit cycles} \right\rangle \subset kQ.
\end{equation}
Dimer algebras originated in string theory \cite{HK, F-W}, and have found wide application to many areas of mathematics (e.g., \cite{BKM, Br, IU, F-V, H, IN}).
Although dimer algebras describe a class of abelian superconformal quiver gauge theories, it is really ghor algebras--and not dimer algebras--that describe abelian quiver gauge theories which are not superconformal.
As we will show, ghor algebras are quotients of superpotential algebras, and the additional relations arise from the assumption that the gauge group is abelian.

We consider the following questions for dimer quivers $Q$ on a torus:
\begin{itemize}
 \item \textit{What is a minimal subset of perfect matchings $\mathcal{P}_0 \subseteq \mathcal{P}$ with the property that the algebra homomorphism
\begin{equation*} \label{P'}
\tau: kQ \to M_n\left(k[ \mathcal{P}_0] \right),
\end{equation*}
defined by (\ref{eta def}) with $\mathcal{P}_0$ in place of $\mathcal{P}$, satisfies $kQ/\ker \eta \cong kQ/ \ker \tau$?}
 \item \textit{How is the ghor algebra $\Lambda$ related to the dimer algebra $A$?}
 \item \textit{How is the center and representation theory of $\Lambda$ related to the perfect matchings in $\mathcal{P}_0$?}
\end{itemize}

Typically $Q$ contains \textit{non-cancellative pairs} of paths, which are distinct paths $p,q$ in $A$ with the property that there is some path $r$ satisfying
$$pr = qr \not = 0 \ \ \ \text{ or } \ \ \ rp = rq \not = 0.$$
If $A$ has a non-cancellative pair, then $A$ and $Q$ are said to be \textit{non-cancellative}, and otherwise $A$ and $Q$ are \textit{cancellative}.
A fundamental characterization of this property is that $A$ is cancellative if and only if $A$ is noetherian \cite[Theorem 1.1]{B3}.

To address our questions, we use the operation of \textit{cyclic contractions} introduced in \cite{B2}, which formalizes and augments the operation of Higgsing in quiver gauge theories.
A cyclic contraction is a $k$-linear map of dimer algebras
$$\psi: A = kQ/I \to A' = kQ'/I',$$
where $Q'$ is cancellative and is obtained by contracting a set of arrows in $Q$ such that the cycles in $Q$ are suitably preserved (Definition \ref{cyclic contraction}).
Cyclic contractions are useful because they allow non-cancellative dimer algebras to be studied by relating them to well-understood cancellative dimer algebras that share similar structure.
Moreover, cyclic contractions exist for all nondegenerate dimer algebras \cite[Theorem 1.1]{B1}.
Examples of cyclic contractions are given in Figure \ref{deformation figure}.

A perfect matching $x \in \mathcal{P}$ is called \textit{simple} if there is a cycle in $Q \setminus x$ that passes through each vertex of $Q$; in this case, $Q \setminus x$ supports a simple $\Lambda$-module (and simple $A$-module) of dimension vector $1^{Q_0}$.
Let $\mathcal{S}$ be the set of simple matchings of $Q$.
We will show that if $Q$ is cancellative, then we may take $\mathcal{P}_0$ to be the simple matchings $\mathcal{S}$.
However, in general $\mathcal{P}_0$ cannot equal $\mathcal{S}$ since there are (nondegenerate) dimer quivers for which $\mathcal{S} = \emptyset$; see Figure \ref{deformation figure}.i.
One may ask if $\mathcal{P}_0$ is simply a minimal set of perfect matchings such that each arrow of $Q$ is contained in some $x \in \mathcal{P}_0$.
To the contrary: for our choice of $\mathcal{P}_0$, there will always be arrows which are not contained in any $x \in \mathcal{P}_0$ whenever $Q$ is non-cancellative.

Our main theorem is the following.

\begin{Theorem} \label{main} (Theorems \ref{first main} and \ref{impression prop}.)
Let $Q$ be a nondegenerate dimer quiver on a torus, and fix a cyclic contraction $\psi: A \to A'$.
Let $\mathcal{P}_0 \subseteq \mathcal{P}$ be the perfect matchings $x$ of $Q$ for which $\psi(x)$ is a simple matching of $Q'$.
Set $B := k[\mathcal{P}_0]$.
\begin{enumerate}
 \item Let $\tau_{\psi}: kQ \to M_n(B)$ be the algebra homomorphism defined in (\ref{eta def}) with $\mathcal{P}_0$ in place of $\mathcal{P}$.
There are algebra isomorphisms
$$\begin{array}{rcl}
\Lambda := kQ/\ker \eta & \cong & kQ/\ker \tau_{\psi} \\
& = & A/\left\langle p-q \ | \ p,q \text{ is a non-cancellative pair} \right\rangle.
\end{array}$$
 \item Suppose $k$ is uncountable.
The induced homomorphism $\tau_{\psi}: \Lambda \to M_n(B)$ classifies all simple $\Lambda$-module isoclasses of maximal $k$-dimension: for each such module $V$, there is a maximal ideal $\mathfrak{b} \in \operatorname{Max}B$ such that
$$V \cong (B/\mathfrak{b})^n,$$
 where $av := \tau_{\psi}(a)v$ for each $a \in \Lambda$, $v \in (B/\mathfrak{b})^n$.
 \item The centers of $\Lambda$ and $\Lambda' \cong A'$ are given by the intersection and union of the vertex corner rings of $\Lambda$,
 $$Z(\Lambda) \cong k\left[ \cap_{i \in Q_0} \bar{\tau}_{\psi}\left( e_i \Lambda e_i \right) \right] \subseteq k\left[ \cup_{i \in Q_0} \bar{\tau}_{\psi}\left( e_i\Lambda e_i \right) \right] \cong Z(\Lambda').$$
\end{enumerate}
\end{Theorem}

To define the ghor algebra $\Lambda$ it thus suffices to keep only those perfect matchings $x \in \mathcal{P}$ for which $\psi(x)$ is a simple matching of $Q'$.

Ghor algebras are contrasted with Broomhead's construction of toric algebras \cite[5.1]{Br}, which are also based on dimer quivers, in Remark \ref{Broomhead}.

\begin{figure}
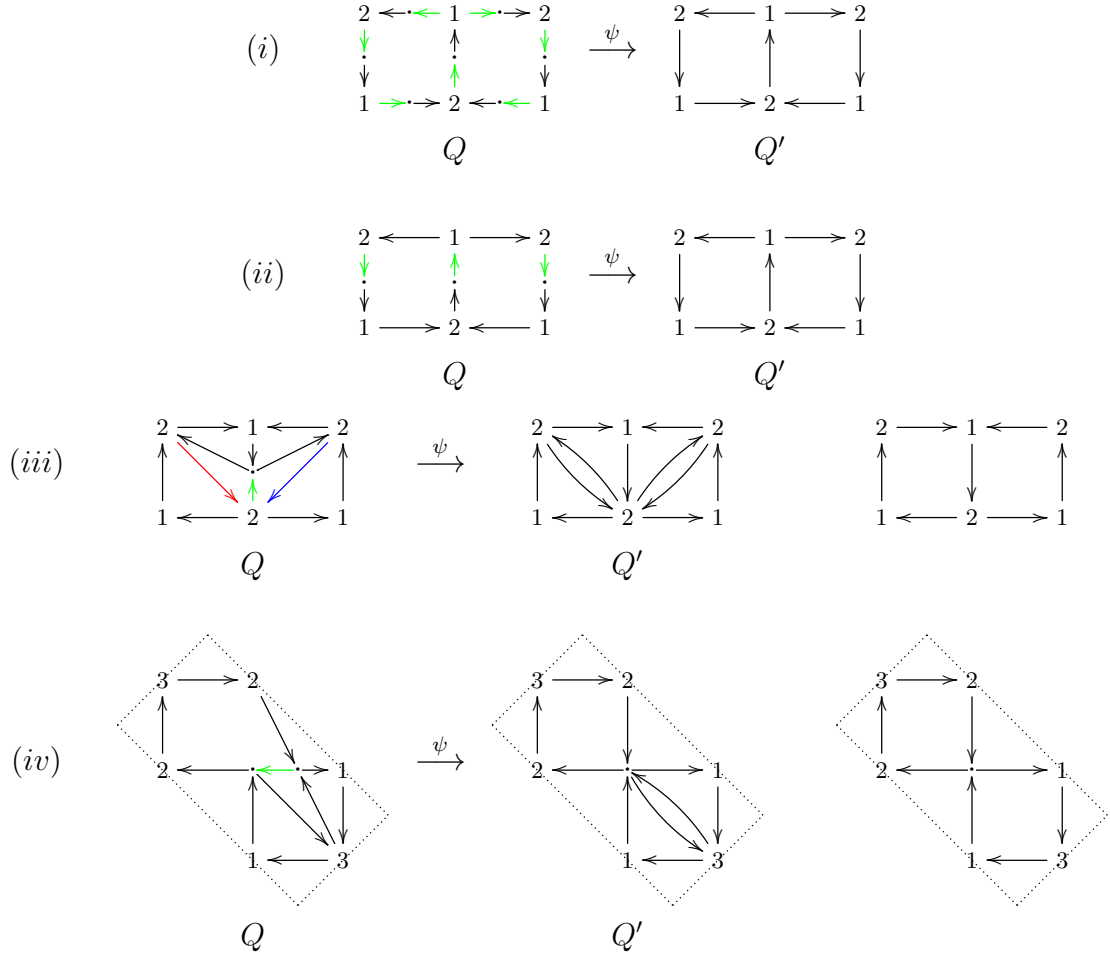

$$\begin{array}{ccccc}
(i) & \ \ & \xy
(-12,6)*+{\text{\scriptsize{$2$}}}="1";(0,6)*+{\text{\scriptsize{$1$}}}="2";(12,6)*+{\text{\scriptsize{$2$}}}="3";
(-12,-6)*+{\text{\scriptsize{$1$}}}="4";(0,-6)*+{\text{\scriptsize{$2$}}}="5";(12,-6)*+{\text{\scriptsize{$1$}}}="6";
(-12,0)*{\cdot}="7";(0,0)*{\cdot}="8";(12,0)*{\cdot}="9";
(-6,6)*{\cdot}="10";(6,6)*{\cdot}="11";(-6,-6)*{\cdot}="12";(6,-6)*{\cdot}="13";
{\ar@[green]"2";"10"};{\ar"10";"1"};{\ar^{}"7";"4"};{\ar@[green]"4";"12"};{\ar"12";"5"};{\ar@[green]"5";"8"};{\ar@[green]"2";"11"};{\ar"11";"3"};{\ar^{}"9";"6"};
{\ar@[green]"6";"13"};{\ar"13";"5"};
{\ar@[green]"3";"9"};{\ar@[green]"1";"7"};{\ar"8";"2"};
\endxy
& \stackrel{\psi}{\longrightarrow} &
\xy
(-12,6)*+{\text{\scriptsize{$2$}}}="1";(0,6)*+{\text{\scriptsize{$1$}}}="2";(12,6)*+{\text{\scriptsize{$2$}}}="3";
(-12,-6)*+{\text{\scriptsize{$1$}}}="4";(0,-6)*+{\text{\scriptsize{$2$}}}="5";(12,-6)*+{\text{\scriptsize{$1$}}}="6";
{\ar^{}"2";"1"};{\ar^{}"1";"4"};{\ar^{}"4";"5"};{\ar^{}"5";"2"};{\ar^{}"2";"3"};{\ar^{}"3";"6"};{\ar^{}"6";"5"}; \endxy \\
&&Q &\ \ \ \ & Q' \\
\\
(ii) & \ \ & \xy
(-12,6)*+{\text{\scriptsize{$2$}}}="1";(0,6)*+{\text{\scriptsize{$1$}}}="2";(12,6)*+{\text{\scriptsize{$2$}}}="3";
(-12,-6)*+{\text{\scriptsize{$1$}}}="4";(0,-6)*+{\text{\scriptsize{$2$}}}="5";(12,-6)*+{\text{\scriptsize{$1$}}}="6";
(-12,0)*{\cdot}="7";(0,0)*{\cdot}="8";(12,0)*{\cdot}="9";
{\ar^{}"2";"1"};{\ar^{}"7";"4"};{\ar^{}"4";"5"};{\ar^{}"5";"8"};{\ar^{}"2";"3"};{\ar^{}"9";"6"};
{\ar^{}"6";"5"};
{\ar@[green]"3";"9"};{\ar@[green]"1";"7"};{\ar@[green]"8";"2"};
\endxy
& \stackrel{\psi}{\longrightarrow} &
\xy
(-12,6)*+{\text{\scriptsize{$2$}}}="1";(0,6)*+{\text{\scriptsize{$1$}}}="2";(12,6)*+{\text{\scriptsize{$2$}}}="3";
(-12,-6)*+{\text{\scriptsize{$1$}}}="4";(0,-6)*+{\text{\scriptsize{$2$}}}="5";(12,-6)*+{\text{\scriptsize{$1$}}}="6";
{\ar^{}"2";"1"};{\ar^{}"1";"4"};{\ar^{}"4";"5"};{\ar^{}"5";"2"};{\ar^{}"2";"3"};{\ar^{}"3";"6"};{\ar^{}"6";"5"};
\endxy
\\
&&Q & \ \ \ \ & Q' \\
\end{array}$$
$$\begin{array}{ccccccc}
(iii) & \ & \xy
(-12,6)*+{\text{\scriptsize{$2$}}}="1";(0,6)*+{\text{\scriptsize{$1$}}}="2";(12,6)*+{\text{\scriptsize{$2$}}}="3";
(-12,-6)*+{\text{\scriptsize{$1$}}}="4";(0,-6)*+{\text{\scriptsize{$2$}}}="5";(12,-6)*+{\text{\scriptsize{$1$}}}="6";
(0,0)*{\cdot}="7";
{\ar^{}"1";"2"};{\ar^{}"2";"7"};{\ar^{}"7";"1"};{\ar@[red]"1";"5"};{\ar^{}"5";"4"};{\ar^{}"4";"1"};{\ar^{}"5";"6"};{\ar"6";"3"};{\ar@[blue]"3";"5"};{\ar^{}"7";"3"};{\ar^{}"3";"2"};
{\ar@[green]"5";"7"};
\endxy
& \stackrel{\psi}{\longrightarrow} &
\xy
(-12,6)*+{\text{\scriptsize{$2$}}}="1";(0,6)*+{\text{\scriptsize{$1$}}}="2";(12,6)*+{\text{\scriptsize{$2$}}}="3";
(-12,-6)*+{\text{\scriptsize{$1$}}}="4";(0,-6)*+{\text{\scriptsize{$2$}}}="5";(12,-6)*+{\text{\scriptsize{$1$}}}="6";
{\ar^{}"1";"2"};{\ar^{}"2";"5"};{\ar^{}"5";"4"};{\ar^{}"4";"1"};{\ar^{}"5";"6"};{\ar^{}"6";"3"};{\ar^{}"3";"2"};{\ar@/_.3pc/"1";"5"};{\ar@/_.3pc/"5";"1"};{\ar@/^.3pc/"5";"3"};{\ar@/^.3pc/"3";"5"};
\endxy
& &
\xy
(-12,6)*+{\text{\scriptsize{$2$}}}="1";(0,6)*+{\text{\scriptsize{$1$}}}="2";(12,6)*+{\text{\scriptsize{$2$}}}="3";
(-12,-6)*+{\text{\scriptsize{$1$}}}="4";(0,-6)*+{\text{\scriptsize{$2$}}}="5";(12,-6)*+{\text{\scriptsize{$1$}}}="6";
{\ar^{}"1";"2"};{\ar^{}"2";"5"};{\ar^{}"5";"4"};{\ar^{}"4";"1"};{\ar^{}"5";"6"};{\ar^{}"6";"3"};{\ar^{}"3";"2"};
\endxy
\\
&& Q & \ \ \ & Q' & \ \ \ &
\\ \\
(iv) & \ & \xy
(-12,12)*+{\text{\scriptsize{$3$}}}="1";(0,12)*+{\text{\scriptsize{$2$}}}="2";(-12,0)*+{\text{\scriptsize{$2$}}}="3";(0,0)*{\cdot}="4";(12,0)*+{\text{\scriptsize{$1$}}}="5";(0,-12)*+{\text{\scriptsize{$1$}}}="6";(12,-12)*+{\text{\scriptsize{$3$}}}="7";
(-18,6)*{}="9";(-6,18)*{}="10";(6,-18)*{}="11";(18,-6)*{}="12";
(6,0)*{\cdot}="8";
{\ar^{}"1";"2"};{\ar^{}"2";"8"};{\ar^{}"4";"3"};{\ar^{}"3";"1"};{\ar^{}"6";"4"};{\ar^{}"4";"7"};{\ar^{}"7";"8"};{\ar^{}"8";"5"};{\ar^{}"5";"7"};{\ar^{}"7";"6"};
{\ar@{..}^{}"9";"10"};{\ar@{..}^{}"10";"12"};{\ar@{..}^{}"12";"11"};{\ar@{..}^{}"11";"9"};
{\ar@[green]"8";"4"};
\endxy
& \stackrel{\psi}{\longrightarrow} &
\xy
(-12,12)*+{\text{\scriptsize{$3$}}}="1";(0,12)*+{\text{\scriptsize{$2$}}}="2";(-12,0)*+{\text{\scriptsize{$2$}}}="3";(0,0)*{\cdot}="4";(12,0)*+{\text{\scriptsize{$1$}}}="5";(0,-12)*+{\text{\scriptsize{$1$}}}="6";(12,-12)*+{\text{\scriptsize{$3$}}}="7";
(-18,6)*{}="9";(-6,18)*{}="10";(6,-18)*{}="11";(18,-6)*{}="12";
{\ar^{}"1";"2"};{\ar^{}"2";"4"};{\ar^{}"4";"3"};{\ar^{}"3";"1"};{\ar^{}"4";"5"};{\ar^{}"5";"7"};{\ar^{}"7";"6"};{\ar^{}"6";"4"};
{\ar@{..}^{}"9";"10"};{\ar@{..}^{}"10";"12"};{\ar@{..}^{}"12";"11"};{\ar@{..}^{}"11";"9"};
{\ar@/_.3pc/"7";"4"};{\ar@/_.3pc/"4";"7"};
\endxy
& &
\xy
(-12,12)*+{\text{\scriptsize{$3$}}}="1";(0,12)*+{\text{\scriptsize{$2$}}}="2";(-12,0)*+{\text{\scriptsize{$2$}}}="3";(0,0)*{\cdot}="4";(12,0)*+{\text{\scriptsize{$1$}}}="5";(0,-12)*+{\text{\scriptsize{$1$}}}="6";(12,-12)*+{\text{\scriptsize{$3$}}}="7";
(-18,6)*{}="9";(-6,18)*{}="10";(6,-18)*{}="11";(18,-6)*{}="12";
{\ar^{}"1";"2"};{\ar^{}"2";"4"};{\ar^{}"4";"3"};{\ar^{}"3";"1"};{\ar^{}"4";"5"};{\ar^{}"5";"7"};{\ar^{}"7";"6"};{\ar^{}"6";"4"};
{\ar@{..}^{}"9";"10"};{\ar@{..}^{}"10";"12"};{\ar@{..}^{}"12";"11"};{\ar@{..}^{}"11";"9"};
\endxy
\\
&&Q & \ \ \ & Q' & \ \ \ &
\end{array}$$
\caption{Some cyclic contractions.
Each quiver is drawn on a torus, and $\psi$ contracts the green arrows.
}
\label{deformation figure}
\end{figure}

\section{Preliminaries}

Let $R$ be an integral domain and a $k$-algebra.
We will denote by $\operatorname{Max}R$ the set of maximal ideals of $R$, and by $\mathcal{Z}(\mathfrak{a})$ the closed set $\left\{ \mathfrak{m} \in \operatorname{Max}R \ | \ \mathfrak{m} \supseteq \mathfrak{a} \right\}$ of $\operatorname{Max}R$ defined by the subset $\mathfrak{a} \subset R$.

We will denote by $Q = \left( Q_0,Q_1,\operatorname{t}, \operatorname{h} \right)$ the quiver with vertex set $Q_0$, arrow set $Q_1$, and head and tail maps $\operatorname{h},\operatorname{t}: Q_1 \to Q_0$.
We will denote by $kQ$ the path algebra of $Q$, and by $e_i$ the idempotent corresponding to vertex $i \in Q_0$.
Multiplication of paths is read right to left, following the composition of maps.
By module we mean left module.
Finally, we will denote by $e_{ij} \in M_d(k)$ the $d \times d$ matrix with a 1 in the $ij$-th slot and zeros elsewhere.

\subsection{Algebra homomorphisms from perfect matchings}

Let $A = kQ/I$ be a dimer algebra on a torus.
Denote by $\mathcal{P}$ and $\mathcal{S}$ the sets of perfect and simple matchings of $Q$ respectively.

Consider the algebra homomorphisms
\begin{equation*} 
\tau: kQ \to M_{|Q_0|}\left(k[ \mathcal{S}]\right) \ \ \ \ \text{ and } \ \ \ \ \eta: kQ \to M_{|Q_0|}\left(k[ \mathcal{P}] \right)
\end{equation*}
defined on $i \in Q_0$ and $a \in Q_1$ by
\begin{align} \label{taua}
\begin{split}
\tau(e_i) = e_{ii}, \ \ \ \ \ \ & \eta(e_i) = e_{ii},\\
\tau(a) = e_{\operatorname{h}(a),\operatorname{t}(a)} \prod_{\substack{x \in \mathcal{S}: \\ x \ni a}} x, \ \ \ \ \ \ & \eta(a) = e_{\operatorname{h}(a),\operatorname{t}(a)} \prod_{\substack{x \in \mathcal{P}: \\ x \ni a}} x,
\end{split}
\end{align}
and extended multiplicatively and $k$-linearly to $kQ$.

\begin{Lemma} \label{tau'A'}
The ideal $I \subset kQ$ given in (\ref{I def}) is contained in the kernels of $\tau$ and $\eta$.
Therefore $\tau$ and $\eta$ induce algebra homomorphisms on the dimer algebra $A$,
\begin{equation*} \label{tau eta}
\tau: A \to M_{|Q_0|}\left(k[ \mathcal{S}]\right) \ \ \ \ \text{ and } \ \ \ \ \eta: A \to M_{|Q_0|}\left(k[ \mathcal{P}] \right).
\end{equation*}
\end{Lemma}

\begin{proof}
Let $p - q$ be a generator for $I$ as given in (\ref{I def}); that is, there is an arrow $a \in Q_1$ such that $pa$ and $qa$ are unit cycles.
Then
$$\tau(p) = e_{\operatorname{t}(a),\operatorname{h}(a)} \prod_{\substack{x \in \mathcal{S}: \\ x \not \ni a}} x = \tau(q).$$
Similarly, $\eta(p) = \eta(q)$.
Therefor $p - q$ is in the kernels of $\tau$ and $\eta$.
\end{proof}

The following lemma is clear.

\begin{Lemma} \label{sigma}
If $\sigma_i, \sigma'_i$ are two unit cycles at $i \in Q_0$, then $\sigma_i = \sigma'_i$ in $A$.
Furthermore, the sum $\sum_{i \in Q_0}\sigma_i$ is in the center of $A$.
\end{Lemma}

We will denote by $\sigma_i \in A$ the unique unit cycle at vertex $i$.

\begin{Lemma}
Each unit cycle $\sigma_i \in e_iAe_i$ satisfies
\begin{equation*} \label{prod D}
\tau\left( \sigma_i \right) = e_{ii} \prod_{x \in \mathcal{S}} x \ \ \ \ \text{ and } \ \ \ \ \eta\left( \sigma_i \right) = e_{ii} \prod_{x \in \mathcal{P}} x.
\end{equation*}
\end{Lemma}

\begin{proof}
Each perfect matching contains precisely one arrow in each unit cycle.
\end{proof}

\begin{Notation} \label{tau bar notation} \rm{
For each $i,j \in Q_0$, denote by
$$\bar{\tau}: e_jAe_i \to B := k[\mathcal{S}] \ \ \ \ \text{ and } \ \ \ \ \bar{\eta}: e_jAe_i \to k[\mathcal{P}]$$
the respective $k$-linear maps defined on $p \in e_jAe_i$ by
$$\tau(p) = \bar{\tau}(p)e_{ji} \ \ \ \ \text{ and } \ \ \ \ \eta(p) = \bar{\eta}(p)e_{ji}.$$
In particular, $\bar{\tau}(p)$ and $\bar{\eta}(p)$ are the single nonzero matrix entries of $\tau(p)$ and $\eta(p)$.
In Section \ref{Cycle structure}, we will set
\begin{equation} \label{A notation}
\overbar{p} := \bar{\tau}(p) \ \ \ \text{ and } \ \ \ \sigma := \prod_{x \in \mathcal{S}} x \ \ \ \text{(or occasionally, $\sigma := \prod_{x \in \mathcal{P}} x$).}
\end{equation}
In Section \ref{Cancellative dimer algebras}, given a cyclic contraction $\psi: A \to A'$ and elements $p \in e_jAe_i$, $q \in e_{\ell}A'e_k$, we will set
\begin{equation*} \label{contraction notation}
\overbar{p} := \bar{\tau}_{\psi}(p) := \bar{\tau}\psi(p), \ \ \ \ \overbar{q} := \bar{\tau}(q), \ \ \ \text{ and } \ \ \ \sigma := \prod_{x \in \mathcal{S}'} x,
\end{equation*}
where $\mathcal{S}'$ is the set of simple matchings of $A'$.
}\end{Notation}

\begin{Remark} \label{Broomhead} \rm{
Let $A = kQ/I$ be a dimer algebra.
In \cite[5.1]{Br}, Broomhead introduced the `toric algebra' $T$ of $Q$:
\begin{equation*}
T := k \left\langle \oplus_{i,j \in Q_0} (\mathbb{Z}^{Q_1}/\partial_2(\underline{1}^{\perp}))^+ \cap \partial_1^{-1}(i-j) \right\rangle,
\end{equation*}
where the maps $\mathbb{Z}^{\operatorname{faces}} \stackrel{\partial_2}{\longrightarrow} \mathbb{Z}^{Q_1} \stackrel{\partial_1}{\longrightarrow} \mathbb{Z}^{Q_0}$ are defined by taking a face to the sum of its boundary arrows, and an arrow $a$ to the difference $\operatorname{h}(a) - \operatorname{t}(a)$.
We remark that this algebra is almost never isomorphic to the ghor algebra $\Lambda$ of $Q$.

Indeed, the toric and ghor algebras of $Q$ are clearly non-isomorphic if $Q$ has no perfect matchings.
So suppose $Q$ is nondegenerate.
Set $n := |Q_0|$ and $\tilde{S} := k[ \cup_{i \in Q_0} \bar{\eta}(e_ikQe_i)] \cong \partial_1^{-1}(0)$.
Then $T$ is isomorphic to the tiled matrix algebra
\begin{equation*}
T \cong \left[ \begin{matrix}
\tilde{S} & \bar{\eta}( e_1kQe_2)\tilde{S} & \cdots & \bar{\eta}(e_1kQe_n)\tilde{S}\\
\bar{\eta}(e_2 kQ e_1)\tilde{S} & \tilde{S} & \cdots & \bar{\eta}(e_2kQe_n)\tilde{S}\\
\vdots & \vdots & \ddots & \vdots \\
\bar{\eta}(e_nkQe_1)\tilde{S} & \bar{\eta}(e_nkQe_2)\tilde{S} & \cdots & \tilde{S}
\end{matrix} \right] 
\subset M_n(k[\mathcal{P}]),
\end{equation*}
where $\tilde{S}$ appears in each diagonal component. 
The ghor algebra $\Lambda$ of $Q$, on the other hand, is the tiled matrix algebra
\begin{equation*}
\Lambda := \left[ \bar{\eta}(e_jkQe_i) \right]_{i,j} = \left[ \begin{matrix}
\bar{\eta}(e_1kQe_1) & \bar{\eta}( e_1kQe_2) & \cdots & \bar{\eta}(e_1kQe_n)\\
\bar{\eta}(e_2 kQ e_1) & \bar{\eta}(e_2kQe_2) & \cdots & \bar{\eta}(e_2kQ e_n)\\
\vdots & \vdots & \ddots & \vdots \\
\bar{\eta}(e_nkQe_1) & \bar{\eta}(e_nkQe_2) & \cdots & \bar{\eta}(e_nkQe_n)
\end{matrix} \right] 
\subset M_n(k[\mathcal{P}]).
\end{equation*}
These two algebras are rarely isomorphic: on a torus, the two algebras coincide if and only if $A$ is cancellative, in which case they are also isomorphic to $A$ (see Theorem 30 below and \cite{Br,Bo}).
On higher genus surfaces, there are only two known exceptional families of dimer quivers whose toric and ghor algebras are isomorphic, given in \cite[Figures 2 and 3]{BB}. 
However, in general the vertex corner rings $\bar{\eta}(e_ikQe_i)$ are not all equal, and therefore the two algebras are not isomorphic. 
Furthermore, the two algebras have very different algebraic and geometric properties; for example, toric algebras are generically noetherian whereas ghor algebras are not \cite{B3, B6}, \cite[Section 4.1]{BB}.
}\end{Remark}

\subsection{Impressions}

The following definition, introduced in \cite{B7}, captures a useful matrix ring embedding.

\begin{Definition} \label{impression definition} \rm{\cite[Definition 2.1]{B7}
Let $A$ be a finitely generated $k$-algebra and let $Z$ be its center.
An \textit{impression} of $A$ is an algebra monomorphism $\tau: A \hookrightarrow M_d(B)$ to a matrix ring over a commutative finitely generated $k$-algebra $B$, such that
\begin{itemize}
 \item for generic $\mathfrak{b} \in \operatorname{Max}B$, the composition
\begin{equation} \label{composition}
A \stackrel{\tau}{\longrightarrow} M_d(B) \stackrel{1}{\longrightarrow} M_d\left(B/\mathfrak{b} \right) \cong M_d(k)
\end{equation}
is surjective; and
 \item the morphism $\operatorname{Max}B \rightarrow \operatorname{Max}\tau(Z)$, $\mathfrak{b} \mapsto \mathfrak{b} \cap \tau(Z)$, is surjective.
\end{itemize}
} \end{Definition}

Surjectivity of (\ref{composition}) implies that the center $Z$ is given by
\begin{equation} \label{Ziso}
Z \cong \left\{ f \in B \ | \ f1_d \in \operatorname{im}\tau \right\} \subseteq B.
\end{equation}
If in addition $A$ is a finitely generated module over its center, then $\tau$ classifies all simple $A$-module isoclasses of maximal $k$-dimension \cite[Proposition 2.5]{B7}.
Specifically, for each such module $V$, there is some $\mathfrak{b} \in \operatorname{Max}B$ such that
\begin{equation} \label{Vcong}
V \cong (B/\mathfrak{b})^d,
\end{equation}
where $av := \tau(a)v$ for each $a \in A$, $v \in (B/\mathfrak{b})^d$.
If $A$ is nonnoetherian, then $\tau$ may characterize the central geometry of $A$ using the framework of depictions \cite[Section 3]{B4}; this relationship is used to study the central geometry of nonnoetherian ghor and dimer algebras in \cite{B6}.

We will show that for a ghor algebra $\Lambda$, the homomorphism $\tau_{\psi}: \Lambda \to M_n(B)$ is an impression (Theorem \ref{impression prop}.1).
Furthermore, a dimer algebra admits an impression if and only if it equals its ghor algebra (Corollary \ref{cr}).

\section{Cyclic contractions} \label{cyclicconractions}

In this section, we describe a new method for studying non-cancellative dimer algebras, first introduced in \cite{B2}, that is based on the notion of Higgsing, or more generally symmetry breaking, in physics.
Using this strategy, we gain information about non-cancellative dimer algebras by relating them to cancellative dimer algebras with similar structure.
Throughout, $A = kQ/I$ is a dimer algebra, typically non-cancellative.

\begin{Definition} \label{contraction} \rm{
Let $Q$ be a dimer quiver, let $Q_1^* \subset Q_1$ be a subset of arrows, and let $Q'$ be the quiver obtained by contracting each arrow in $Q_1^*$.
Specifically, $Q'$ is formed by simultaneously removing each arrow $a$ in $Q_1^*$, and merging together the head and tail vertices of $a$.
This operation defines a $k$-linear map of path algebras
$$\psi: kQ \rightarrow kQ',$$
where
$$\psi(a) = \left\{ \begin{array}{cl} a & \text{ if } \ a \in Q_0 \cup Q_1 \setminus Q_1^* \\ e_{\operatorname{t}(a)} & \text{ if } \ a \ \in Q_1^* \end{array} \right.$$
and extended multiplicatively to paths and $k$-linearly to $kQ$.
If $\psi$ induces a $k$-linear map of dimer algebras
$$\psi: A = kQ/I \to A' = kQ'/I',$$
that is, $\psi(I) \subseteq I'$, then we call $\psi$ a \textit{contraction of dimer algebras}.
}\end{Definition}

We now describe the structure we wish to preserve under a contraction.
To specify this structure, we introduce the following commutative algebras.

\begin{Definition} \label{cyclic contraction} \rm{
Let $\psi: A \to A'$ be a contraction to a cancellative dimer algebra.
If
$$S := k\left[ \cup_{i \in Q_0} \bar{\tau}\psi\left(e_iAe_i\right) \right] = k\left[ \cup_{i \in Q'_0} \bar{\tau}\left( e_iA'e_i \right) \right],$$
then we say $\psi$ is \textit{cyclic}, and call $S$ the \textit{cycle algebra} of $A$.
}\end{Definition}

The cycle algebra is independent of the choice of $\psi$ by \cite[Theorem 3.13]{B2}.

\begin{Example} \rm{
Four cyclic contractions are given in Figure \ref{deformation figure}.
The non-cancellative quivers $Q$ in (ii) and (iv) have appeared in the physics literature (e.g., \cite[Section 4]{F-R}, \cite{FKR}; and \cite[Table 6, 2.6]{DHP}).
The unit 2-cycles in (iii) and (iv) consist of arrows that are redundant generators for $A' = kQ'/I'$, and so may be removed from $Q'$.
In (iii), let $a,b,c$ be the respective red, blue, and green arrows in $Q$.
Observe that in $A = kQ/I$, we have
$$ab \not = ba \ \ \ \text{ and } \ \ \ cab = cba.$$
Thus the pair $ab$, $ba$ is non-cancellative (in fact, $a$ and $b$ generate a free subalgebra of $A$).
In (i), $Q$ has no simple matchings.
These examples are considered in more detail in Example \ref{four examples}.
}\end{Example}

\begin{Notation} \rm{
For $g,h \in B$, by $g \mid h$ we mean $g$ divides $h$ in $B$, even if $g$ or $h$ is assumed to be in $S$.
}\end{Notation}

\begin{Lemma} \label{cannot contract}
Suppose $\psi: A \to A'$ is a contraction of dimer algebras, and $A'$ has a perfect matching.
Then $\psi$ cannot contract a unit cycle of $A$ to a vertex.
\end{Lemma}

\begin{proof}
Assume to the contrary that $\psi$ contracts the unit cycle $\sigma_j \in A$ to the vertex $e_{\psi(j)} \in A'$.
Fix a unit cycle $\sigma_{i'} \in A'$.
Since $\psi$ is surjective on $Q'_0$, there is a vertex $i \in Q_0$ such that $\psi(i) = i'$.
Let $p \in A$ be a path from $i$ to $j$, and set $p' := \psi(p)$.
Then
\begin{equation} \label{p' sigma}
p' \sigma_{i'} \stackrel{\textsc{(i)}}{=} \psi(p \sigma_i) \stackrel{\textsc{(ii)}}{=} \psi(\sigma_j p) \stackrel{\textsc{(iii)}}{=} \psi(\sigma_j) p' = e_{\psi(j)}p' = p'.
\end{equation}
Indeed, (\textsc{i}) and (\textsc{iii}) hold by Definition \ref{contraction}, and (\textsc{ii}) holds by Lemma \ref{sigma}.

Denote by $\mathcal{P}'$ the set of perfect matchings of $A'$.
Set $\sigma := \prod_{x \in \mathcal{P}'}x$.
Then (\ref{p' sigma}) implies
$$\bar{\eta}(p') \sigma = \bar{\eta}(p' \sigma_{i'}) = \bar{\eta}(p') \in k[\mathcal{P}' ],$$
by Lemma \ref{tau'A'}.
Whence $\sigma = 1$.
But this contradicts our assumption that $\mathcal{P}' \not = \emptyset$.
\end{proof}

An example where a unit cycle is contracted to a vertex is given in Figure \ref{non-example unit}.

\begin{figure}
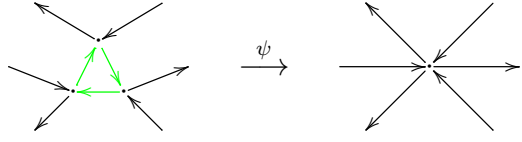

$$\xy 0;/r.4pc/:
(-5,5)*{}="1";(5,5)*{}="2";(7.07,0)*{}="3";(5,-5)*{}="4";(-5,-5)*{}="5";(-7.07,0)*{}="6";
(0,2)*{\cdot}="7";(2,-2)*{\cdot}="8";(-2,-2)*{\cdot}="9";
{\ar"7";"1"};{\ar"2";"7"};{\ar"4";"8"};{\ar"6";"9"};{\ar"9";"5"};{\ar"8";"3"};
{\ar@[green]"7";"8"};{\ar@[green]"8";"9"};{\ar@[green]"9";"7"};\endxy
\ \ \ \ \stackrel{\psi}{\longrightarrow} \ \ \ \
\xy 0;/r.4pc/:
(-5,5)*{}="1";(5,5)*{}="2";(7.07,0)*{}="3";(5,-5)*{}="4";(-5,-5)*{}="5";(-7.07,0)*{}="6";(0,0)*{\cdot}="7";
{\ar"7";"1"};{\ar"2";"7"};{\ar"4";"7"};{\ar"7";"5"};{\ar"6";"7"};{\ar"7";"3"};\endxy$$
\caption{The contraction of a unit cycle (drawn in green) to a vertex.  Such a contraction cannot induce a contraction of dimer algebras $\psi: A \to A'$ if $A'$ has a perfect matching, by Lemma \ref{cannot contract}.}
\label{non-example unit}
\end{figure}

\begin{Lemma} \label{positive length cycle}
Suppose $\psi: A \to A'$ is a contraction of dimer algebras, and $A'$ has a perfect matching.
Then $\psi$ cannot contract a cycle in the underlying graph $\overbar{Q}$ of $Q$ to a vertex.
In particular,
\begin{enumerate}
 \item $\psi$ cannot contract a cycle in $Q$ to a vertex;
 \item if $p$ is a cycle in $Q$ that is not null-homotopic as a loop on the torus, then $\psi(p)$ is a cycle in $Q'$ that is also not null-homotopic; and
 \item $A$ does not have a non-cancellative pair where one of the paths is a vertex.
\end{enumerate}
\end{Lemma}

\begin{proof}
The number of vertices, edges, and faces in the underlying graphs $\overbar{Q}$ and $\overbar{Q}'$ of $Q$ and $Q'$ are given by
$$\begin{array}{lclcl}
V = \left| Q_0 \right|, & \ \ \ & E = \left| Q_1 \right|, & \ \ \ & F =  \# \text{ of connected components of } T^2 \setminus \overbar{Q}, \\
V' = \left| Q'_0 \right|, & & E' = \left| Q'_1 \right|, & & F' =  \# \text{ of connected components of } T^2 \setminus \overbar{Q}'.
\end{array}$$
Since $\overbar{Q}$ and $\overbar{Q}'$ each embed into a two-torus, their respective Euler characteristics vanish:
\begin{equation*} \label{VEF}
V - E + F = 0, \ \ \ \ V' - E' + F' = 0.
\end{equation*}

Assume to the contrary that $\psi$ contracts the cycles $p_1,\ldots, p_{\ell}$ in $\overbar{Q}$ to vertices in $\overbar{Q}'$.
Denote by $n_0$ and $n_1$ the respective number of vertices and arrows in $Q$ which are subpaths of some $p_i$, $1 \leq i \leq \ell$.
Denote by $m$ the number of vertices in $Q'_0$ of the form $\psi(p_i)$ for some $1 \leq i \leq \ell$.
By assumption, $m \geq 1$.

In any cycle, the number of trivial subpaths equals the number of arrow subpaths.
Furthermore, if two cycles share a common edge, then they also share a common vertex.
Therefore
$$n_1 \geq n_0.$$
Whence
$$\begin{array}{rcl}
0 & = & F' - E' + V'\\
& = & F' - (E - n_1) + (V - n_0 + m)\\
& = & F' + (-E + V) + (n_1 - n_0) + m\\
& \geq & F' - F + m.
\end{array}$$
Thus $F' < F$ since $m \geq 1$.
Therefore $\psi$ contracts a face of $Q$ to a vertex.
In particular, some unit cycle in $Q$ is contracted to a vertex.
But this is not possible by Lemma \ref{cannot contract}.
\end{proof}

\begin{Remark} \rm{
Suppose $\psi: A \to A'$ is a cyclic contraction.
We will show in Lemma \ref{at least one} below that $A'$, being cancellative, necessarily has a perfect matching.
Therefore Lemmas \ref{cannot contract} and \ref{positive length cycle} hold in the case $\psi$ is cyclic.
}\end{Remark}

An immediate question is whether all non-cancellative dimer algebras admit cyclic contractions.
If $Q$ is nondegenerate, then $A$ admits a cyclic contraction \cite[Theorem 1.1]{B1}.
However, there are degenerate dimer algebras that do not admit contractions (cyclic or not) to cancellative dimer algebras.
For example, dimer algebras that contain permanent 2-cycles (Definition \ref{removable2}) cannot contract to cancellative dimer algebras, by Lemma \ref{positive length cycle}.1.

A similar question is whether cyclic contractions can exist between two different cancellative dimer algebras.
The answer to this question is negative: if $\psi: A \to A'$ is a cyclic contraction and $A$ is cancellative, then $\psi$ must be the identity map \cite[Theorem 1.1]{B3}.

\begin{Remark} \rm{
Suppose $\psi: A \to A'$ is a contraction.
Consider a path $p = a_n \cdots a_2a_1 \in kQ$ with $a_1, \ldots, a_n \in Q_0 \cup Q_1$.
If $p \not = 0$, then by definition
$$\psi(p) = \psi(a_n) \cdots \psi(a_1) \in kQ'.$$
However, we claim that if $\psi$ is nontrivial, then it is not an algebra homomorphism.
Indeed, let $\delta \in Q_1^*$.
Consider a path $a_2 \delta a_1 \not = 0$ in $A$.
By Lemma \ref{positive length cycle}.1, $\delta$ is not a cycle.
In particular, $\operatorname{h}(a_1) \not = \operatorname{t}(a_2)$.
Whence
$$\psi(a_2a_1) = \psi(0) = 0 \not = \psi(a_2)\psi(a_1).$$
We note, however, that the restriction
$$\psi: \epsilon_0A\epsilon_0 \to A' \ \ \ \text{ where } \ \epsilon_0 := 1_A - \sum_{\delta \in Q_1^*}e_{\operatorname{h}(\delta)},$$
is an algebra homomorphism \cite[Proposition 2.12.1]{B2}.
}\end{Remark}

\begin{Remark} \rm{
The notion of Higgsing in quiver gauge theories (that is, contracting a set of arrows to vertices) was established in the early 2000's, before dimer algebras appeared (e.g., \cite{F-H}).
Higgsing was first introduced in the context of Morita equivalences by the author in 2013 in \cite[arXiv v1]{B2}, and a couple of months later by Ishii and Ueda in \cite[arXiv v2; Morita equivalences did not appear in v1]{IU}.
The latter was based on Gulotta's `inverse algorithm' from 2008 \cite{G}.
Gulotta's algorithm produces cancellative dimer quivers for arbitrary polygons and uses Higgsing in an essential way. 
In contrast, both the cycle algebra and the operation of Higgsing \textit{that preserves the cycle algebra}, that is, cyclic contractions, were introduced in \cite{B2} to study non-cancellative dimer algebras.
Cyclic contractions have subsequently played a fundamental role in the articles \cite{B3, B5, B6}.
}\end{Remark}

\section{Cycle structure} \label{Cycle structure}

Let $A = kQ/I$ be a dimer algebra, possibly with no perfect matchings.
Unless stated otherwise, by path or cycle we mean path or cycle modulo $I$.
Throughout, we use the notation (\ref{A notation}).

\begin{Notation} \rm{
Let $\pi: \mathbb{R}^2 \rightarrow T^2$ be a covering map such that for some $i \in Q_0$,
$$\pi\left(\mathbb{Z}^2 \right) = i.$$
Denote by $Q^+ := \pi^{-1}(Q) \subset \mathbb{R}^2$ the covering quiver of $Q$.
For each path $p$ in $Q$, denote by $p^+$ the unique path in $Q^+$ with tail in the unit square $[0,1) \times [0,1) \subset \mathbb{R}^2$ satisfying $\pi(p^+) = p$.

For paths $p$, $q$ satisfying
\begin{equation} \label{exception}
\operatorname{t}(p^+) = \operatorname{t}(q^+) \ \ \text{ and } \ \ \operatorname{h}(p^+) = \operatorname{h}(q^+),
\end{equation}
denote by $\mathcal{R}_{\tilde{p},\tilde{q}}$ the compact region in $\mathbb{R}^2 \supset Q^+$ bounded by representatives $\tilde{p}^+$, $\tilde{q}^+$ of $p^+$, $q^+$.\footnote{In Lemma \ref{columns and pillars}.1, we will not require (\ref{exception}) to hold.}
If the representatives are fixed (or arbitrary), we will write $\mathcal{R}_{p,q}$ for $\mathcal{R}_{\tilde{p},\tilde{q}}$.
}\end{Notation}

\begin{Definition} \label{non-cancellative pair def} \rm{ \
We say a non-cancellative pair $p,q \in A$ is \textit{minimal} if for each non-cancellative pair $s,t \in A$, we have  $\{s,t\} = \{ p,q \}$ whenever there exists representatives $\tilde{p}, \tilde{q}, \tilde{s},\tilde{t}$ satisfying
$$\mathcal{R}_{\tilde{s},\tilde{t}} \subseteq \mathcal{R}_{\tilde{p},\tilde{q}}.$$
} \end{Definition}

\begin{Lemma} \label{here2}
Let $p,q \in e_jAe_i$ be distinct paths such that
\begin{equation*} \label{p+}
\operatorname{t}(p^+) = \operatorname{t}(q^+) \ \ \text{ and } \ \ \operatorname{h}(p^+) = \operatorname{h}(q^+).
\end{equation*}
The following hold.
\begin{enumerate}
 \item $p \sigma_i^m = q \sigma_i^n$ for some $m, n \geq 0$.
 \item $\bar{\tau}(p) = \bar{\tau}(q) \sigma^m$ and $\bar{\eta}(p) = \bar{\eta}(q) \sigma^m$ for some $m \in \mathbb{Z}$.
 \item If $p,q$ is a non-cancellative pair, then $\bar{\tau}(p) = \bar{\tau}(q)$ and $\bar{\eta}(p) = \bar{\eta}(q)$.
 \item If $\mathcal{P} \not = \emptyset$ and $\bar{\eta}(p) = \bar{\eta}(q)$, then $p,q$ is a non-cancellative pair.
\end{enumerate}
\end{Lemma}

\begin{proof}
(1) We proceed by induction on the region $\mathcal{R}_{p,q} \subset \mathbb{R}^2$ bounded by $p^+$ and $q^+$, with respect to inclusion.

If there are unit cycles $sa$ and $ta$ with $a \in Q_1$, and $\mathcal{R}_{p,q} = \mathcal{R}_{s,t}$, then the claim is clear.
So suppose the claim holds for all pairs of paths $s,t$ such that
$$\operatorname{t}(s^+) = \operatorname{t}(t^+), \ \ \ \operatorname{h}(s^+) = \operatorname{h}(t^+), \ \ \text{ and } \ \ \mathcal{R}_{s,t} \subset \mathcal{R}_{p,q}.$$

Factor $p$ and $q$ into subpaths,
$$p = p_m p_{m-1} \cdots p_1 \ \ \text{ and } \ \ q = q_n q_{n-1} \cdots q_1,$$
such that for each $1 \leq \alpha \leq m$ and $1 \leq \beta \leq n$, there are paths $p'_{\alpha}$ and $q'_{\beta}$ for which
$$p'_{\alpha}p_{\alpha} \ \ \text{ and } \ \ q'_{\beta}q_{\beta}$$
are unit cycles, and ${p'_{\alpha}}^+$ and ${q'_{\beta}}^+$ lie in the region $\mathcal{R}_{p,q}$.
See Figure \ref{purple}.
Note that if $p_{\alpha}$ is itself a unit cycle, then $p'_{\alpha}$ is the vertex $\operatorname{t}(p_{\alpha})$, and similarly for $q_{\beta}$.

Consider the paths
$$p' := p'_1p'_2 \cdots p'_m \ \ \text{ and } \ \ q' := q'_1q'_2 \cdots q'_n.$$
Then by Lemma \ref{sigma} there is some $c,d \geq 0$ such that
\begin{equation*} \label{p'p =}
p'p = \sigma_i^c \ \ \text{ and } \ \ q'q = \sigma_i^d.
\end{equation*}

Now
$$\operatorname{t}(p'^+) = \operatorname{t}(q'^+), \ \ \ \operatorname{h}(p'^+) = \operatorname{h}(q'^+), \ \ \text{ and } \ \ \mathcal{R}_{p',q'} \subset \mathcal{R}_{p,q}.$$
Thus, by induction, there is some $c',d' \geq 0$ such that
$$p' \sigma_i^{c'} = q' \sigma_i^{d'}.$$
Therefore by Lemma \ref{sigma},
$$p \sigma_i^{d+d'} = p q'q \sigma_i^{d'} = pq'\sigma_j^{d'}q = pp'\sigma_j^{c'}q = p p' q \sigma_i^{c'} = q \sigma_i^{c+c'},$$
proving our claim.

(2) By Claim (1) there is some $m, n \geq 0$ such that
\begin{equation} \label{sigmaj}
p \sigma^m_i = q \sigma_i^n.
\end{equation}
Thus by Lemma \ref{tau'A'},
\begin{equation} \label{sigmaj2}
\bar{\eta}(p)\sigma^m = \bar{\eta}\left( p\sigma^m_i \right) = \bar{\eta}\left( q\sigma_i^n \right) = \bar{\eta}(q)\sigma^n \in B.
\end{equation}
Claim (2) then follows since $B$ is an integral domain.
A similar result holds for $\bar{\tau}$ in place of $\bar{\eta}$ by setting each non-simple perfect matching $x \in \mathcal{P}$ equal to $1$ in (\ref{sigmaj2}).  

(3) Suppose $p,q$ is a non-cancellative pair.
Then there is a path $r$ such that
$$rp = rq \not = 0 \ \ \ \text{ or } \ \ \ pr = qr \not = 0;$$
say $rp = rq$.
Whence
$$\bar{\eta}(r)\bar{\eta}(p) = \bar{\eta}(rp) = \bar{\eta}(rq) = \bar{\eta}(r)\bar{\eta}(q),$$
by Lemma \ref{tau'A'}.
Therefore $\bar{\eta}(p) = \bar{\eta}(q)$ since $B$ is an integral domain.
Similarly, $\bar{\tau}(p) = \bar{\tau}(q)$.

(4) Finally, suppose $\mathcal{P} \not = \emptyset$ and $\bar{\eta}(p) = \bar{\eta}(q)$.
Set $\sigma := \prod_{x \in \mathcal{P}} x$.
Then (\ref{sigmaj2}) implies
$$\sigma^m = \sigma^n$$
since $B$ is an integral domain (with $\bar{\eta}$ in place of $\bar{\tau}$).
By assumption, $\mathcal{P} \not = \emptyset$.
Whence $\sigma \not = 1$.
Therefore $m = n$.
Consequently, the path
$$r = \sigma_i^m$$
satisfies $pr = qr \not = 0$ by (\ref{sigmaj}).
\end{proof}

\begin{figure}
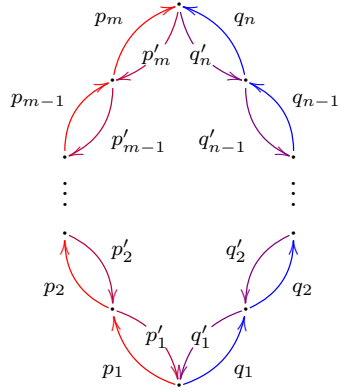

$$\xy 0;/r.3pc/:
(0,-20)*{\cdot}="1";(-7,-12)*{\cdot}="2";(-12,-4)*{\cdot}="3";
(-12,4)*{\cdot}="4";(-7,12)*{\cdot}="5";(0,20)*{\cdot}="6";
(7,-12)*{\cdot}="7";(12,-4)*{\cdot}="8";
(12,4)*{\cdot}="9";(7,12)*{\cdot}="10";
(-12,1)*{\vdots}="";(12,1)*{\vdots}="";
{\ar@/^/^{p_1}@[red]"1";"2"};{\ar@/^/^{p_2}@[red]"2";"3"};
{\ar@/^/^{p_{m-1}}@[red]"4";"5"};{\ar@/^/^{p_m}@[red]"5";"6"};
{\ar@/_/_{q_1}@[blue]"1";"7"};{\ar@/_/_{q_2}@[blue]"7";"8"};
{\ar@/_/_{q_{n-1}}@[blue]"9";"10"};{\ar@/_/_{q_n}@[blue]"10";"6"};
{\ar@/^/|-{p'_1}@[purple]"2";"1"};{\ar@/^/^{p'_2}@[purple]"3";"2"};
{\ar@/^/^{p'_{m-1}}@[purple]"5";"4"};{\ar@/^/|-{p'_m}@[purple]"6";"5"};
{\ar@/_/|-{q'_1}@[violet]"7";"1"};{\ar@/_/_{q'_2}@[violet]"8";"7"};
{\ar@/_/_{q'_{n-1}}@[violet]"10";"9"};{\ar@/_/|-{q'_n}@[violet]"6";"10"};
\endxy$$
\caption{Setup for Lemma \ref{here2}.1.  The paths $p = p_m \cdots p_1$, $q = q_n \cdots q_1$, $p' = p' _1\cdots p'_m$, and $q' = q'_1 \cdots q'_n$ are drawn in red, blue, purple, and violet respectively.  Each product $p'_{\alpha}p_{\alpha}$ and $q'_{\beta}q_{\beta}$ is a unit cycle.  Note that the region $\mathcal{R}_{p',q'}$ is properly contained in the region $\mathcal{R}_{p,q}$.}
\label{purple}
\end{figure}

\newpage

\begin{Lemma} \label{r+} \
\begin{enumerate}
 \item Suppose paths $p,q$ are either equal modulo $I$, or form a non-cancellative pair.
Then their lifts $p^+$ and $q^+$ bound a compact region $\mathcal{R}_{p,q}$ in $\mathbb{R}^2$.
 \item Suppose paths $p,q$ are equal modulo $I$.
If $i^+$ is a vertex in $\mathcal{R}_{p,q}$, then there is a path $r^+$ from $\operatorname{t}(p^+)$ to $\operatorname{h}(p^+)$ that is contained in $\mathcal{R}_{p,q}$, passes through $i^+$, and satisfies
\begin{equation*} \label{r++}
p = r = q \ \ \text{ (modulo $I$)}.
\end{equation*}
\end{enumerate}
\end{Lemma}

\begin{proof}
(1.i) First suppose $p,q$ are equal modulo $I$.
The relations generated by $I$ in (\ref{I def}) lift to homotopy relations on the paths of $Q^+$.
Thus $\operatorname{t}(p^+) = \operatorname{t}(q^+)$ and $\operatorname{h}(p^+) = \operatorname{h}(q^+)$.
Therefore $p^+$ and $q^+$ bound a compact region in $\mathbb{R}^2$.

(1.ii) Now suppose $p,q$ is a non-cancellative pair.
Then there is a path $r$ such that $rp = rq \not = 0$, say.
In particular, $\operatorname{t}((rp)^+) = \operatorname{t}((rq)^+)$ and $\operatorname{h}((rp)^+) = \operatorname{h}((rq)^+)$ by Claim (1.i.).
Therefore $\operatorname{t}(p^+) = \operatorname{t}(q^+)$ and $\operatorname{h}(p^+) = \operatorname{h}(q^+)$ as well.

(2) The ideal $I$ is generated by relations of the form $s - t$, where there is an arrow $a$ such that $sa$ and $ta$ are unit cycles.
The claim then follows since each trivial subpath of the unit cycle $sa$ (resp.\ $ta$) is a trivial subpath of $s$ (resp.\ $t$).
\end{proof}

\begin{Definition} \label{removable2} \rm{
A unit cycle $\sigma_i \in A$ of length 2 is a \textit{removable 2-cycle} if the two arrows it is composed of are redundant generators for $A$, and otherwise $\sigma_i$ is a \textit{permanent 2-cycle}.
}\end{Definition}

\begin{Lemma} \label{permanent 2-cycles}
There are precisely two types of permanent 2-cycles, given in Figures \ref{2-cycle}.ii and \ref{2-cycle}.iii.
Consequently, if a dimer algebra $A$ has a permanent 2-cycle, then $A$ is degenerate.
\end{Lemma}

\begin{proof}
Let $ab$ be a permanent 2-cycle, with $a,b \in Q_1$.
Let
$$\sigma_{\operatorname{t}(a)} = sa \ \ \text{ and } \ \ \sigma'_{\operatorname{t}(b)} = tb$$
be the complementary unit cycles to $ab$ containing $a$ and $b$ respectively.
Since $ab$ is permanent, $b$ is a subpath of $s$, or $a$ is a subpath of $t$.

Suppose $b$ is a subpath of $s$.
Then there are (possibly trivial) paths $p,q$ such that $s = qbp$.
We therefore have either case (ii) (if $q$ is null-homotopic as a loop) or case (iii) (if $q$ is not null-homotopic) shown in Figure \ref{2-cycle}.
\end{proof}

\begin{figure}
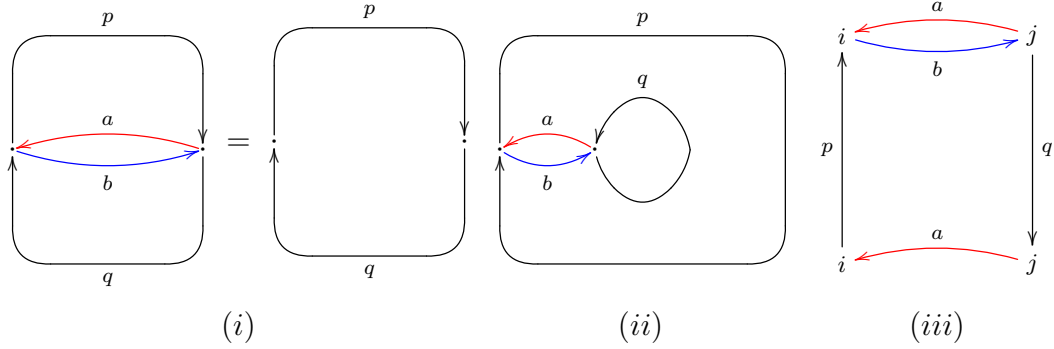

$$\begin{array}{ccc}
\xy 0;/r.6pc/:
(-5,0)*{\cdot}="1";(-5,4)*{}="2";
(-3,6)*{}="3";(3,6)*{}="4";
(5,4)*{}="5";(5,0)*{\cdot}="6";(5,-4)*{}="7";
(3,-6)*{}="8";(-3,-6)*{}="9";
(-5,-4)*{}="10";
{\ar@{-}"1";"2"};{\ar@{-}@/^.45pc/"2";"3"};{\ar@{-}^p"3";"4"};{\ar@{-}@/^.45pc/"4";"5"};{\ar@{->}"5";"6"};
{\ar@{-}"6";"7"};{\ar@{-}@/^.45pc/"7";"8"};{\ar@{-}^q"8";"9"};{\ar@{-}@/^.45pc/"9";"10"};{\ar@{->}"10";"1"};
{\ar@/_/_a@[red]"6";"1"};{\ar@/_/_b@[blue]"1";"6"};
\endxy
\ = \
\xy 0;/r.6pc/:
(-5,0)*{\cdot}="1";(-5,4)*{}="2";
(-3,6)*{}="3";(3,6)*{}="4";
(5,4)*{}="5";(5,0)*{\cdot}="6";(5,-4)*{}="7";
(3,-6)*{}="8";(-3,-6)*{}="9";
(-5,-4)*{}="10";
{\ar@{-}"1";"2"};{\ar@{-}@/^.45pc/"2";"3"};{\ar@{-}^p"3";"4"};{\ar@{-}@/^.45pc/"4";"5"};{\ar@{->}"5";"6"};
{\ar@{-}"6";"7"};{\ar@{-}@/^.45pc/"7";"8"};{\ar@{-}^q"8";"9"};{\ar@{-}@/^.45pc/"9";"10"};{\ar@{->}"10";"1"};
\endxy
&
\xy 0;/r.6pc/:
(-8,0)*{\cdot}="1";(-8,4)*{}="2";
(-6,6)*{}="3";(5,6)*{}="4";
(7,4)*{}="5";(7,-4)*{}="7";
(5,-6)*{}="8";(-6,-6)*{}="9";
(-8,-4)*{}="10";(-3,0)*{\cdot}="6";
(2,0)*{}="11";
(1,-2)*{}="12";(1,2)*{}="13";
(-2,-2)*{}="14";(-2,2)*{}="15";
{\ar@{-}"1";"2"};{\ar@{-}@/^.45pc/"2";"3"};{\ar@{-}^p"3";"4"};{\ar@{-}@/^.45pc/"4";"5"};{\ar@{-}"5";"7"};
{\ar@{-}@/^.45pc/"7";"8"};{\ar@{-}"8";"9"};{\ar@{-}@/^.45pc/"9";"10"};{\ar@{->}"10";"1"};
{\ar@/_/_a@[red]"6";"1"};{\ar@/_/_b@[blue]"1";"6"};
{\ar@{-}@/_.45pc/"14";"12"};{\ar@{-}@/_.45pc/_q"13";"15"};
{\ar@{-}@/_.1pc/"12";"11"};{\ar@{-}@/_.1pc/"11";"13"};
{\ar@/_.1pc/"15";"6"};{\ar@{-}@/_.1pc/"6";"14"};
\endxy
&
\xy 0;/r.6pc/:
(-5,-6)*+{\text{\scriptsize{$i$}}}="1";(5,-6)*+{\text{\scriptsize{$j$}}}="2";(-5,6)*+{\text{\scriptsize{$i$}}}="3";(5,6)*+{\text{\scriptsize{$j$}}}="4";
{\ar^p"1";"3"};{\ar^q"4";"2"};
{\ar@/_/_a@[red]"4";"3"};{\ar@/_/_b@[blue]"3";"4"};
{\ar@/_/_a@[red]"2";"1"};
\endxy
\\
(i) & (ii) & (iii)
\end{array}$$
\caption{Cases for Lemma \ref{permanent 2-cycles}.  In each case, $a$ and $b$ are arrows, and $p$ and $q$ are paths.  In case (i) $ap$, $bq$, $ab$ are unit cycles; in cases (ii) and (iii) $qbpa$, $ab$ are unit cycles, and $p$, $q$ are cycles.
(In case (ii), $q$ may be a vertex, and $q$ need not be a unit cycle, that is, $q^+$ may contain arrows in its interior.)
In case (i) $ab$ is a removable 2-cycle, and in cases (ii) and (iii) $ab$ is a permanent 2-cycle.}
\label{2-cycle}
\end{figure}

\begin{Notation} \rm{
By a \textit{cyclic subpath} of a path $p$, we mean a subpath of $p$ that is a nontrivial cycle.
Consider the following sets of cycles in $A$:
\begin{itemize}
 \item Let $\mathcal{C}$ be the set of cycles in $A$ (i.e., cycles in $Q$ modulo $I$).
 \item For $u \in \mathbb{Z}^2$, let $\mathcal{C}^u$ be the set of cycles $p \in \mathcal{C}$ such that
$$\operatorname{h}(p^+) = \operatorname{t}(p^+) + u \in Q_0^+.$$
 \item For $i \in Q_0$, let $\mathcal{C}_i$ be the set of cycles in the vertex corner ring $e_iAe_i$.
 \item Let $\hat{\mathcal{C}}$ be the set of cycles $p \in \mathcal{C}$ such that the lift of each cyclic permutation of each representative of $p$ does not have a cyclic subpath.
\end{itemize}
We decorate $\mathcal{C}$ so as to specify a set of cycles; e.g., $\hat{\mathcal{C}}^u_i := \hat{\mathcal{C}} \cap \mathcal{C}^u \cap \mathcal{C}_i$.  
Note that $\mathcal{C}^0 := \mathcal{C}^{(0,0)}$ is the set of cycles whose lifts are cycles in $Q^+$.
In particular, $\hat{\mathcal{C}}^0 = Q_0$. 
Furthermore, the lift of any cycle $p$ in $\hat{\mathcal{C}}$ has no cyclic subpaths, although $p$ itself may have cyclic subpaths.
}\end{Notation}

\begin{Lemma} \label{here3'}
Suppose $A$ does not have a non-cancellative pair where one of the paths is a vertex.
Let $p$ be a nontrivial cycle.
\begin{enumerate}
 \item If $p \in \mathcal{C}^0$, then $\overbar{p} = \sigma^m$ for some $m \geq 1$.
 \item If $p \in \mathcal{C}^0$ and $A$ is cancellative, then $p = \sigma_{\operatorname{t}(p)}^m$ for some $m \geq 1$.
 \item If $p \in \mathcal{C} \setminus \hat{\mathcal{C}}$, then $\sigma \mid \overbar{p}$.
 \item If $p$ is a path for which $\sigma \nmid \overbar{p}$, then $p$ is a subpath of a cycle in $\hat{\mathcal{C}}$.
\end{enumerate}
\end{Lemma}

\begin{proof}
(1) Suppose $p \in \mathcal{C}^0$, that is, $p^+$ is a cycle in $Q^+$.
Set $i := \operatorname{t}(p)$.
By Lemma \ref{here2}.1, there is some $m,n \geq 0$ such that
\begin{equation} \label{n > m}
p\sigma_i^m = \sigma_i^n.
\end{equation}
If $m \geq n$, then the paths $p \sigma_i^{m-n}$ and $e_{\operatorname{t}(p)}$ form a non-cancellative pair.
But this is contrary to assumption.
Thus $n - m \geq 1$.

Furthermore, $\bar{\tau}$ is an algebra homomorphism on $e_iAe_i$ by Lemma \ref{tau'A'}.
In particular, (\ref{n > m}) implies
$$\overbar{p}\sigma^m = \sigma^n.$$
Therefore $\overbar{p} = \sigma^{n - m}$ since $B$ is an integral domain.

(2) If $A$ is cancellative, then (\ref{n > m}) implies $p = \sigma_i^{n-m}$.

(3) Suppose $p \in \mathcal{C} \setminus \hat{\mathcal{C}}$.
Then there is a cyclic permutation of $p^+$ which contains a cyclic subpath $q^+$.
In particular, $\overbar{q} \mid \overbar{p}$.
Furthermore, since $q \in \mathcal{C}^0$, Claim (1) implies $\overbar{q} = \sigma^m$ for some $m \geq 1$.
Therefore $\sigma \mid \overbar{p}$.

(4) Let $p$ be a path for which $\sigma \nmid \overbar{p}$.
Then there is a simple matching $x \in \mathcal{S}$ such that $x \nmid \overbar{p}$.
In particular, $p$ is supported on $Q \setminus x$.
Since $x$ is simple, $p$ is a subpath of a cycle $q$ supported on $Q \setminus x$.
Whence $x \nmid \overbar{q}$.
Thus $\sigma \nmid \overbar{q}$.
Therefore $q$ is in $\hat{\mathcal{C}}$ by the contrapositive of Claim (3).
\end{proof}

\begin{Remark} \label{weird} \rm{
In Lemma \ref{here3'} we assumed that $A$ does not have a non-cancellative pair where one of the paths is a vertex.
Such non-cancellative pairs exist: consider the permanent 2-cycle in Figure \ref{2-cycle}.ii, and suppose $q$ is a trivial cycle. 
Then $p\sigma_{\operatorname{t}(p)} = pab = \sigma_{\operatorname{t}(p)}$.
In particular, the cycles $p, e_{\operatorname{t}(p)}$ form a non-cancellative pair.

Furthermore, it is possible for $m > n$ in (\ref{n > m}).
Consider a dimer algebra with the subquiver given in Figure \ref{weird fig}.
Here $a,b,c,d$ are arrows, $q$ is a nontrivial path, and $ad$, $bc$, $qdcba$ are unit cycles.
Let $m \geq 1$, and set $p := baq^{m+1}dc$.
Then
$$p \sigma_i^m = baq^{m+1}dc \sigma_i^m \stackrel{\textsc{(i)}}{=} ba(q \sigma_j)^m qdc \stackrel{\textsc{(ii)}}{=} b(aqd)^{m+1}c \stackrel{\textsc{(iii)}}{=} bc = \sigma_i.$$
Indeed, (\textsc{i}) holds by Lemma \ref{sigma}; (\textsc{ii}) holds since $\sigma_j = da$; and (\textsc{iii}) holds since $b = baqd$ in $A$.
Note, however, that such a dimer quiver does not admit a perfect matching.
}\end{Remark}

\begin{figure}
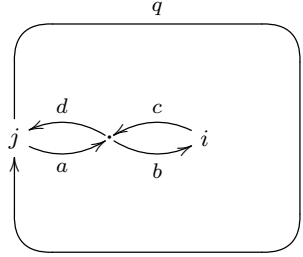

$$\xy 0;/r.6pc/:
(-8,0)*+{\text{\scriptsize{$j$}}}="1";(-8,4)*{}="2";
(-6,6)*{}="3";(5,6)*{}="4";
(7,4)*{}="5";(7,-4)*{}="7";
(5,-6)*{}="8";(-6,-6)*{}="9";
(-8,-4)*{}="10";
(-3,0)*{\cdot}="11";
(2,0)*+{\text{\scriptsize{$i$}}}="6";
{\ar@{-}"1";"2"};{\ar@{-}@/^.45pc/"2";"3"};{\ar@{-}^q"3";"4"};{\ar@{-}@/^.45pc/"4";"5"};{\ar@{-}"5";"7"};
{\ar@{-}@/^.45pc/"7";"8"};{\ar@{-}"8";"9"};{\ar@{-}@/^.45pc/"9";"10"};{\ar@{->}"10";"1"};
{\ar@/_/_c"6";"11"};{\ar@/_/_d"11";"1"};
{\ar@/_/_a"1";"11"};{\ar@/_/_b"11";"6"};
\endxy$$
\caption{Setup for Remark \ref{weird}.
Let $m \geq 1$.
Then the path $p := baq^{m+1}dc$ satisfies $p \sigma_i^m = \sigma_i$.}
\label{weird fig}
\end{figure}

\begin{Remark} \rm{
Some of the following results in this section that relate simple matchings, homotopy of paths, and cancellativity can be obtained by combining certain results of several articles \cite{MR, D, Br, Bo, IU}.
The proofs we give here, however, are new, independent, and self-contained, and are based on different methods. 

In particular, it was shown by Broomhead in \cite[Proposition 6.2]{Br} that if $A$ is a `geometrically consistent' dimer algebra, then for each $i, j \in Q_0^+$ there is a path $p^+$ from $i$ to $j$ such that, in our notation, $\overbar{p}$ is not divisible by some `extremal' (or `corner') perfect matching $x$.
This result was proceeded by similar results by Mozgovoy and Reineke in \cite[Sections 3-5]{MR}, and subsequently extended to cancellative dimer algebras by Bocklandt in \cite[Theorem 8.7]{Bo}. 
The latter built on the work of Davison \cite[Definition 2.5]{D}, which introduced cancellativity but did not consider perfect matchings.
It was then shown by Ishii and Ueda in \cite[Proposition 8.2]{IU}, a few years later, that simple matchings and extremal matchings coincide.
Putting these three results together we obtain Proposition \ref{here} below.

Similarly, a few of the lemmas in this section that involve simple matchings, namely Lemmas \ref{longlist}, \ref{easy injective}, \ref{generated by}, \ref{finally!'}, and Proposition \ref{circle 1} (see \cite[Proposition 1.2]{Br}), also follow from combining various results of these articles.
We emphasize, however, that simple matchings are defined quite differently from extremal matchings, and were not considered in \cite{Br}, \cite{Bo}, \cite{MR}, or \cite{D}.
These lemmas therefore do not follow directly from any of one these articles.
}\end{Remark}

For the remainder of this section, we assume that $Q$ has at least one perfect matching, $\mathcal{P} \not = \emptyset$, unless stated otherwise.

By the definition of $\hat{\mathcal{C}}$, each cycle $p \in \mathcal{C}^u_i \setminus \hat{\mathcal{C}}$ has a representative $\tilde{p}$ that factors into subpaths $\tilde{p} = \tilde{p}_3 \tilde{p}_2 \tilde{p}_1$, where
\begin{equation} \label{C0}
p_1p_3 = \tilde{p}_1 \tilde{p}_3 + I \in \mathcal{C}^0 \ \ \text{ and } \ \ p_2 = \tilde{p}_2 + I \in \hat{\mathcal{C}}^u.
\end{equation}

\begin{Proposition} \label{here}
If $\hat{\mathcal{C}}^u_i = \emptyset$ for some $u \in \mathbb{Z}^2$ and $i \in Q_0$, then $A$ is non-cancellative.
\end{Proposition}

\begin{proof}
In the following, we will use the notation (\ref{C0}).
Set $\sigma := \prod_{x \in \mathcal{P}} x$.

Suppose $\hat{\mathcal{C}}^u_i = \emptyset$.
Then $u$ is in $\mathbb{Z}^2 \setminus 0$ since $\hat{\mathcal{C}}^0_i \not = \emptyset$ for each $i \in Q_0$. 
Let $p,q \in \mathcal{C}^u_i$ be cycles such that the region
$$\mathcal{R}_{\tilde{p}_3\tilde{p}_2^2\tilde{p}_1,\tilde{q}_3\tilde{q}_2^2\tilde{q}_1}$$
contains the vertex $i^+ + u \in Q_0^+$.
Furthermore, suppose $p$ and $q$ admit representatives $\tilde{p}'$ and $\tilde{q}'$ (possibly distinct from $\tilde{p}$ and $\tilde{q}$) such that the region $\mathcal{R}_{\tilde{p}',\tilde{q}'}$ has minimal area among all such pairs of cycles $p,q$.
See Figure \ref{6figure}.

By Lemmas \ref{tau'A'} and \ref{here2}.1, there is some $m \in \mathbb{Z}$ such that
$$\bar{\eta}(p_3p_2^2p_1) = \bar{\eta}(q_3q_2^2q_1) \sigma^m.$$
Suppose $m \geq 0$.
Set
$$s := p_3p_2^2p_1 \ \ \text{ and } \ \ t := q_3q_2^2q_1 \sigma_i^m.$$
Then
\begin{equation*} \label{s == t}
\bar{\eta}(s) = \bar{\eta}(t).
\end{equation*}

Assume to the contrary that $A$ is cancellative.
If $s \not = t$, then $s,t$ would be a non-cancellative pair by Lemma \ref{here2}.4.
Therefore $s = t$.
Furthermore, there is a path $r^+$ in $\mathcal{R}_{\tilde{s},\tilde{t}}$ which passes through the vertex
$$i^+ + u \in Q_0^+$$
and is homotopic to $s^+$ (by the relations $I$), by Lemma \ref{r+}.2.
In particular, $r$ factors into paths $r = r_2e_ir_1 = r_2r_1$, where
$$r_1, r_2 \in \mathcal{C}^u_i.$$
But $p$ and $q$ were chosen so that the area of $\mathcal{R}_{\tilde{p}',\tilde{q}'}$ is minimal.
Thus there is some $\ell_1, \ell_2 \geq 0$ such that
$$\tilde{r}_1 = \tilde{p}'\sigma^{\ell_1}_i \ \ \text{ and } \ \ \tilde{r}_2 = \tilde{p}'\sigma^{\ell_2}_i \ \ \ \text{(modulo $I$)}.$$
Set $\ell := \ell_1 + \ell_2$.
Then
\begin{equation} \label{r p2}
r = r_2r_1 = p^2 \sigma^{\ell_1 + \ell_2}_i = p^2 \sigma^{\ell}_i.
\end{equation}

Since $A$ is cancellative, the $\bar{\eta}$-image of any nontrivial cycle in $Q^+$ is a positive power of $\sigma$ by Lemma \ref{here3'}.1 (with $\bar{\eta}$ in place of $\bar{\tau}$).
In particular, since $(p_1p_3)^+$ is a nontrivial cycle, there is an $n \geq 1$ such that
\begin{equation} \label{p1p3}
\bar{\eta}(p_1p_3) = \sigma^n.
\end{equation}

Hence
$$\bar{\eta}(p)\bar{\eta}(p_2) \stackrel{\textsc{(i)}}{=} \bar{\eta}(s) = \bar{\eta}(r) \stackrel{\textsc{(ii)}}{=} \bar{\eta}(p^2)\sigma^{\ell} \stackrel{\textsc{(iii)}}{=} \bar{\eta}(p)\bar{\eta}(p_2) \bar{\eta}(p_1p_3) \sigma^{\ell} \stackrel{\textsc{(iv)}}{=} \bar{\eta}(p)\bar{\eta}(p_2) \sigma^{n + \ell}.$$
Indeed, (\textsc{i}) and (\textsc{iii}) hold by Lemma \ref{tau'A'}, (\textsc{ii}) holds by (\ref{r p2}), and (\textsc{iv}) holds by (\ref{p1p3}).
Thus, since $k[\mathcal{P}]$ is an integral domain, we have
$$\sigma^{n + \ell} = 1.$$
But $n \geq 1$ and $\ell \geq 0$.
Whence $\sigma = 1$.
Therefore $Q$ has no perfect matchings, contrary to our standing assumption that $Q$ has at least one perfect matching.
\end{proof}

\begin{figure}
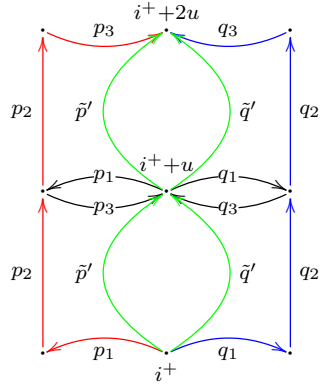

$$\xy 0;/r.3pc/:
(0,-17)*{\cdot}="1";
(13,-17)*{\cdot}="2";(0,0)*{\cdot}="3";(13,0)*{\cdot}="4";(0,17)*{\cdot}="5";
(13,17)*{\cdot}="6";
(-13,-17)*{\cdot}="7";(-13,0)*{\cdot}="8";(-13,17)*{\cdot}="9";
{\ar_{q_2}@[blue]"2";"4"};{\ar_{q_2}@[blue]"4";"6"};{\ar^{p_2}@[red]"7";"8"};{\ar^{p_2}@[red]"8";"9"};
{\ar@/^/_{q_1}@[blue]"1";"2"};{\ar@/_/^{p_1}@[red]"1";"7"};
{\ar@/^/|-{q_1}"3";"4"};{\ar@/_/|-{p_1}"3";"8"};
{\ar@/^/|-{q_3}"4";"3"};{\ar@/_/|-{p_3}"8";"3"};
{\ar@/^/_{q_3}@[blue]"6";"5"};{\ar@/_/^{p_3}@[red]"9";"5"};
(0,-17)*{}="10";(0,1)*{}="11";(0,17)*{}="12";
{\ar@{}_{i^+}"10";"10"};
{\ar@{}^{i^+ + u}"11";"11"};
{\ar@{}^{i^+ + 2u}"12";"12"};
{\ar@/_2pc/_{\tilde{q}'}@[green]"1";"3"};
{\ar@/^2pc/^{\tilde{p}'}@[green]"1";"3"};
{\ar@/_2pc/_{\tilde{q}'}@[green]"3";"5"};
{\ar@/^2pc/^{\tilde{p}'}@[green]"3";"5"};
\endxy$$
\caption{Setup for Proposition \ref{here}, drawn on the cover $Q^+$.  
The (lifts of the) cycles $s = p_3p_2^2p_1$ and $q_3q_2^2q_1$ are drawn in red and blue respectively.  The representatives $\tilde{p}'$ and $\tilde{q}'$ of $p$ and $q$, such that the region $\mathcal{R}_{\tilde{p}',\tilde{q}'}$ has minimal area, are drawn in green.}
\label{6figure}
\end{figure}

\begin{Definition} \rm{
We call the subquiver given in Figure \ref{figure3}.i a \textit{column}, and the subquiver given in Figure \ref{figure3}.ii a \textit{pillar}.
In the latter case, $\operatorname{h}(a_{\ell})$ and $\operatorname{t}(a_1)$ are either trivial subpaths of $q_{\ell}$ and $p_1$ respectively, or $p_{\ell}$ and $q_1$ respectively.
}\end{Definition}

\begin{figure}
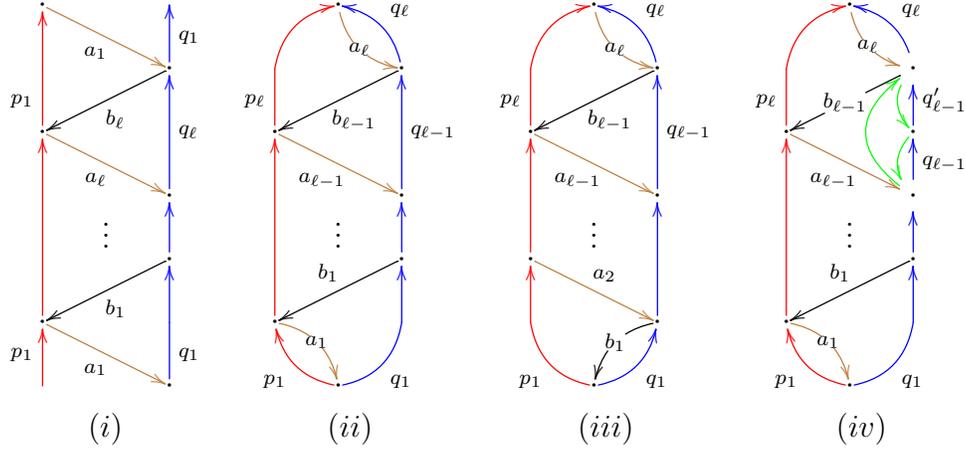

$$\begin{array}{cccc}
\xy 0;/r.4pc/:
(-5,-15)*{}="1";(5,-15)*{\cdot}="1'";
(-5,-10)*{\cdot}="2";(5,-5)*{\cdot}="3";
(0,-2.5)*{\vdots}="";
(5,0)*{\cdot}="6";(-5,5)*{\cdot}="5";
(5,10)*{\cdot}="7";
(-5,15)*{\cdot}="8";(5,15)*{}="8'";
(5,-10)*{}="9";(-5,10)*{}="10";
{\ar^{p_1}@[red]"1";"2"};{\ar_{a_1}@[brown]"2";"1'"};{\ar@{-}_{q_1}@[blue]"1'";"9"};
{\ar^{}@[blue]"9";"3"};{\ar^{b_1}"3";"2"};{\ar^{}@[red]"2";"5"};{\ar^{}@[blue]"3";"6"};
{\ar_{q_{\ell}}@[blue]"6";"7"};{\ar^{b_{\ell}}"7";"5"};{\ar@{-}^{p_1}@[red]"5";"10"};{\ar^{}@[red]"10";"8"};
{\ar_{a_1}@[brown]"8";"7"};{\ar_{q_1}@[blue]"7";"8'"};{\ar_{a_{\ell}}@[brown]"5";"6"};
\endxy
&
\xy 0;/r.4pc/:
(0,-15)*{\cdot}="1";(-5,-10)*{\cdot}="2";(5,-5)*{\cdot}="3";
(0,-2.5)*{\vdots}="";
(5,0)*{\cdot}="6";(-5,5)*{\cdot}="5";
(5,10)*{\cdot}="7";(0,15)*{\cdot}="8";(5,-10)*{}="9";(-5,10)*{}="10";
{\ar@/^/^{p_1}@[red]"1";"2"};{\ar@/^/|-{a_1}@[brown]"2";"1"};{\ar@{-}@/_/_{q_1}@[blue]"1";"9"};
{\ar^{}@[blue]"9";"3"};{\ar_{b_1}"3";"2"};{\ar^{}@[red]"2";"5"};{\ar^{}@[blue]"3";"6"};
{\ar_{q_{\ell-1}}@[blue]"6";"7"};{\ar^{b_{\ell-1}}"7";"5"};{\ar@{-}^{p_{\ell}}@[red]"5";"10"};{\ar@/^/^{}@[red]"10";"8"};
{\ar@/_/|-{a_{\ell}}@[brown]"8";"7"};{\ar@/_/_{q_{\ell}}@[blue]"7";"8"};{\ar_{a_{\ell-1}}@[brown]"5";"6"};
\endxy
&
\xy 0;/r.4pc/:
(0,-15)*{\cdot}="1";(-5,-10)*{}="2";(5,-5)*{}="3";
(0,-2.5)*{\vdots}="";
(5,0)*{\cdot}="6";(-5,5)*{\cdot}="5";
(5,10)*{\cdot}="7";(0,15)*{\cdot}="8";(5,-10)*{\cdot}="9";(-5,10)*{}="10";
(-5,-5)*{\cdot}="11";
{\ar@{-}@/^/^{p_1}@[red]"1";"2"};{\ar@/_/|-{b_1}"9";"1"};{\ar@/_/_{q_1}@[blue]"1";"9"};
{\ar^{}@[blue]"9";"6"};{\ar^{a_2}@[brown]"11";"9"};{\ar^{}@[red]"2";"11"};{\ar^{}@[red]"11";"5"};
{\ar_{q_{\ell-1}}@[blue]"6";"7"};{\ar^{b_{\ell-1}}"7";"5"};{\ar@{-}^{p_{\ell}}@[red]"5";"10"};{\ar@/^/^{}@[red]"10";"8"};
{\ar@/_/|-{a_{\ell}}@[brown]"8";"7"};{\ar@/_/_{q_{\ell}}@[blue]"7";"8"};{\ar_{a_{\ell-1}}@[brown]"5";"6"};
\endxy
&
\xy 0;/r.4pc/:
(0,-15)*{\cdot}="1";(-5,-10)*{\cdot}="2";(5,-5)*{\cdot}="3";
(0,-2.5)*{\vdots}="";
(5,0)*+{\cdot}="6";(-5,5)*{\cdot}="5";
(5,10)*+{\cdot}="7";(0,15)*{\cdot}="8";(5,-10)*{}="9";(-5,10)*{}="10";
(5,5)*{\cdot}="11";
{\ar@/^/^{p_1}@[red]"1";"2"};{\ar@/^/|-{a_1}@[brown]"2";"1"};{\ar@{-}@/_/_{q_1}@[blue]"1";"9"};
{\ar^{}@[blue]"9";"3"};{\ar_{b_1}"3";"2"};{\ar^{}@[red]"2";"5"};{\ar^{}@[blue]"3";"6"};
{\ar_{q_{\ell-1}}@[blue]"6";"11"};
{\ar_{q'_{\ell-1}}@[blue]"11";"7"};
{\ar@/_/@[green]"7";"11"};{\ar@/_/@[green]"11";"6"};{\ar@/^1.5pc/@[green]"6";"7"};
{\ar|-{b_{\ell-1}}"7";"5"};{\ar@{-}^{p_{\ell}}@[red]"5";"10"};{\ar@/^/^{}@[red]"10";"8"};
{\ar@/_/|-{a_{\ell}}@[brown]"8";"7"};{\ar@/_/_{q_{\ell}}@[blue]"7";"8"};{\ar_{a_{\ell-1}}@[brown]"5";"6"};
\endxy
\\
(i) & (ii) & (iii) & (iv)
\end{array}$$
\caption{Cases for Lemmas \ref{columns and pillars} and \ref{figure3lemma}.
The subquiver in case (i) is a `column', and the subquiver in case (ii) is a `pillar'.
The paths $a_1, \ldots, a_{\ell}, b_1, \ldots, b_{\ell}$ are arrows, and the paths $p_{\ell} \cdots p_1$ and $q_{\ell} \cdots q_1$, drawn in red and blue respectively, are not-necessarily-proper subpaths of $p$ and $q$.
Case (iv) is an example with extra interior arrows, drawn in green.
In cases (i), (ii), (iii), the cycles $q_ja_jb_j$ and $a_jp_jb_{j-1}$ are unit cycles.
Note that in cases (iii) and (iv), $q^+$ has a cyclic subpath (modulo $I$), contrary to assumption.
Furthermore, $pb_{\ell} = b_{\ell}q$ in case (i), and $p_{\ell} \cdots p_1 = q_{\ell} \cdots q_1$ in case (ii).}
\label{figure3}
\end{figure}

\begin{figure}
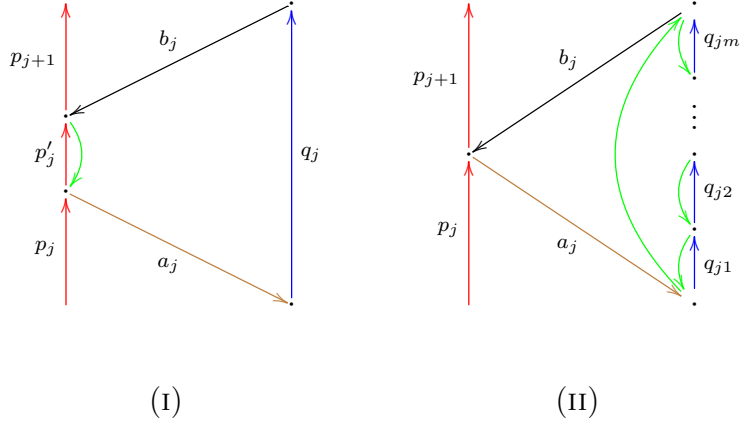

$$\begin{array}{ccc}
\xy
(-15,-20)*{}="1";(15,-20)*{\cdot}="2";(-15,20)*{}="3";(15,20)*{\cdot}="4";(-15,-5)*{\cdot}="5";(-15,5)*{\cdot}="6";
{\ar_{q_j}@[blue]"2";"4"};{\ar_{b_j}"4";"6"};{\ar@/^/@[green]"6";"5"};{\ar_{a_j}@[brown]"5";"2"};{\ar^{p_j}@[red]"1";"5"};{\ar^{p'_j}@[red]"5";"6"};{\ar^{p_{j+1}}@[red]"6";"3"};
\endxy
& \ \ \ &
\xy
(-15,-20)*{}="1'";
(-15,0)*{\cdot}="6";(-15,20)*{}="3'";
{\ar^{p_j}@[red]"1'";"6"};{\ar^{p_{j+1}}@[red]"6";"3'"};
(15,-20)*+{\cdot}="1";(15,-10)*{\cdot}="2";(15,0)*{\cdot}="3";(15,10)*{\cdot}="4";(15,20)*+{\cdot}="5";
(15,6)*{\vdots}="";
{\ar_{b_j}"5";"6"};{\ar_{a_j}@[brown]"6";"1"};{\ar@/^2.5pc/@[green]"1";"5"};
{\ar@/_/@[green]"5";"4"};{\ar@/_/@[green]"3";"2"};{\ar@/_/@[green]"2";"1"};
{\ar_{q_{j1}}@[blue]"1";"2"};{\ar_{q_{j2}}@[blue]"2";"3"};{\ar_{q_{jm}}@[blue]"4";"5"};
\endxy
\\ \\
(\textsc{i}) & & (\textsc{ii})
\end{array}$$
\caption{Generalizations of the setup given in Figure \ref{figure3}.iv.
In both cases, the green paths are arrows.
Case (\textsc{i}) shows that if $\operatorname{t}(a_j) \not = \operatorname{h}(b_j)$ (resp.\ $\operatorname{h}(a_j) \not = \operatorname{t}(b_{j-1})$), then $p^+$ (resp.\ $q^+$) has a cyclic subpath.
In case (\textsc{ii}), $q_j = q_{jm}\cdots q_{j2}q_{j1}$.
This case shows that if $q_ja_jb_j$ (resp.\ $a_jp_jb_{j-1}$) is not a unit cycle, then $q^+$ (resp.\ $p^+$) has a cyclic subpath.}
\label{factor2}
\end{figure}

\begin{Lemma} \label{columns and pillars}
Suppose paths $p^+$, $q^+$ have no cyclic subpaths (modulo $I$), and bound a region $\mathcal{R}_{p,q}$ that contains no vertices in its interior.
\begin{enumerate}
 \item If $p$ and $q$ do not intersect, then $p^+$ and $q^+$ bound a column.
 \item Otherwise $p^+$ and $q^+$ bound a union of pillars.
In particular, if
$$\operatorname{t}(p^+) = \operatorname{t}(q^+) \ \ \ \text{ and } \ \ \ \operatorname{h}(p^+) = \operatorname{h}(q^+) \not = \operatorname{t}(p^+),$$
then $p = q$ (modulo $I$).
\end{enumerate}
\end{Lemma}

\begin{proof}
Since $\mathcal{R}_{p,q}$ contains no vertices in its interior, each path that intersects its interior is an arrow.
Thus $p^+$ and $q^+$ bound a union of subquivers given by the four cases in Figure \ref{figure3}.

In case (i),
$$p = p_{\ell} \cdots p_1 \ \ \ \text{ and } \ \ \ q = q_{\ell} \cdots q_1.$$
In cases (ii) - (iv), the paths $p_{\ell} \cdots p_1$ and $q_{\ell} \cdots q_1$ are not-necessarily-proper subpaths of $p$ and $q$ respectively.
In cases (ii) and (iv), $\operatorname{h}(a_{\ell})$ and $\operatorname{t}(a_1)$ are either trivial subpaths of $q$ and $p$ respectively, or $p$ and $q$ respectively.
In case (iii), $\operatorname{h}(a_{\ell})$ and $\operatorname{t}(a_1)$ are either both trivial subpaths of $q$, or both trivial subpaths of $p$.
In all cases, each cycle which bounds a region with no arrows in its interior is a unit cycle.
In particular, in cases (i) - (iii), each path $a_jp_jb_{j-1}$, $b_jq_ja_j$ is a unit cycle.

Observe that in cases (iii) and (iv), $q^+$ has a cyclic subpath, contrary to assumption.
Generalizations of case (iv) are considered in Figures \ref{factor2}.\textsc{i} and \ref{factor2}.\textsc{ii}.
Consequently, $p^+$ and $q^+$ bound either a column or a union of pillars.
\end{proof}

\begin{Notation} \label{non-crossing} \rm{
For $u \in \mathbb{Z}^2$, denote by $\widehat{\mathcal{C}}^u$ a maximal set of representatives of cycles in $\hat{\mathcal{C}}^u$ whose lifts do not intersect transversely in $\mathbb{R}^2$ (though they may share common subpaths).
}\end{Notation}

In the following two lemmas, fix $u \in \mathbb{Z}^2$ and a subset $\widehat{\mathcal{C}}^u$.

\begin{Lemma} \label{figure2lemma}
Suppose $\hat{\mathcal{C}}^u_i \not = \emptyset$ for each $i \in Q_0$.
Then $\widehat{\mathcal{C}}^u_i \not = \emptyset$ for each $i \in Q_0$.
\end{Lemma}

\begin{proof}
If $u = 0$, then $\widehat{\mathcal{C}}^u_i = \hat{\mathcal{C}}^u_i = \left\{ e_i \right\}$.
So suppose $u \in \mathbb{Z}^2 \setminus 0$.

Let $\tilde{p},\tilde{q}$ be representatives of cycles $p,q$ in $\hat{\mathcal{C}}^u$ that intersect transversely at $k \in Q_0$.
Then their lifts $\tilde{p}^+$ and $\tilde{q}^+$ intersect at least twice since $p$ and $q$ are both in $\mathcal{C}^u$.
Thus $\tilde{p}$ and $\tilde{q}$ factor into paths $\tilde{p} = p_3p_2p_1$ and $\tilde{q} = q_3q_2q_1$, where
$$\operatorname{t}(p_2^+) = \operatorname{t}(q_2^+) \ \ \text{ and } \ \ \operatorname{h}(p_2^+) = \operatorname{h}(q_2^+),$$
and $k^+$ is the tail or head of $p_2^+$.
See Figure \ref{figure2}.

Since $\hat{\mathcal{C}}^u_i \not = \emptyset$ for each $i \in Q_0$, we may suppose that $\mathcal{R}_{p_2,q_2}$ contains no vertices in its interior.
Since $p$ and $q$ are in $\hat{\mathcal{C}}$, $p_2^+$ and $q_2^+$ do not have cyclic subpaths.
Thus $p_2 = q_2$ (modulo $I$) by Lemma \ref{columns and pillars}.2.
Therefore the paths
$$\tilde{s} := p_3q_2p_1 \ \ \ \text{ and } \ \ \ \tilde{t} := q_3p_2q_1$$
equal $\tilde{p}$ and $\tilde{q}$ (modulo $I$) respectively.
In particular, $s$ and $t$ are in $\hat{\mathcal{C}}^u$.
The lemma then follows since $\tilde{s}^+$ and $\tilde{t}^+$ do not intersect transversely at the vertices $\operatorname{t}(p_2^+)$ or $\operatorname{h}(p_2^+)$.
\end{proof}

\begin{figure}
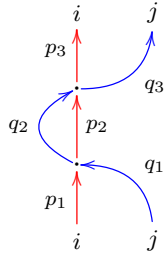

$$\xy 0;/r.4pc/:
(-3,-9)*+{\text{\scriptsize{$i$}}}="1";(-3,-3)*{\cdot}="2";(-3,3)*{\cdot}="3";(-3,9)*+{\text{\scriptsize{$i$}}}="4";
(3,-9)*+{\text{\scriptsize{$j$}}}="5";(3,9)*+{\text{\scriptsize{$j$}}}="6";
{\ar^{p_1}@[red]"1";"2"};{\ar_{p_2}@[red]"2";"3"};{\ar^{p_3}@[red]"3";"4"};
{\ar@/_.9pc/_{q_1}@[blue]"5";"2"};{\ar@/^1.2pc/^{q_2}@[blue]"2";"3"};{\ar@/_.9pc/_{q_3}@[blue]"3";"6"};
\endxy$$
\caption{Setup for Lemma \ref{figure2lemma}.}
\label{figure2}
\end{figure}

\begin{Lemma} \label{figure3lemma}
Let $u \in \mathbb{Z}^2 \setminus 0$, and suppose $\hat{\mathcal{C}}^u_i \not = \emptyset$ for each $i \in Q_0$.
Then there is a simple matching $x \in \mathcal{S}$ such that $Q \setminus x$ supports each cycle in $\widehat{\mathcal{C}}^u$.

Furthermore, if $A$ contains a column, then there are two simple matchings $x_1, x_2 \in \mathcal{S}$ such that $Q \setminus (x_1 \cup x_2)$ supports each cycle in $\widehat{\mathcal{C}}^u$.
\end{Lemma}

\begin{proof}
By Lemma \ref{figure2lemma}, $\widehat{\mathcal{C}}^u_i \not = \emptyset$ for each $i \in Q_0$ since $\hat{\mathcal{C}}^u_i \not = \emptyset$ for each $i \in Q_0$.
Thus we may consider cycles $p,q \in \widehat{\mathcal{C}}^u$ for which $\pi^{-1}(p)$ and $\pi^{-1}(q)$ bound a region $\mathcal{R}_{p,q}$ with no vertices in its interior.

Recall Figure \ref{figure3}.
Since $\widehat{\mathcal{C}}^u_i \not = \emptyset$ for each $i \in Q_0$, we may partition $Q$ into columns and pillars, by Lemma \ref{columns and pillars}.
Consider the subset of arrows $x_1$ (resp.\ $x_2$) consisting of all the $a_j$ arrows in each pillar of $Q$, and all the $a_j$ arrows (resp.\ $b_j$ arrows) in each column of $Q$.
Note that $x_1$ consists of all the right-pointing arrows in each column, and $x_2$ consists of all the left-pointing arrows in each column.
Furthermore, if $Q$ does not contain a column, then $x_1 = x_2$.

Observe that each unit cycle in each column and pillar contains precisely one arrow in $x_1$, and one arrow in $x_2$.
(Note that this is not true for cases (iii) and (iv).)
Furthermore, no such arrow occurs on the boundary of these regions, that is, as a subpath of $p$ or $q$.
Therefore $x_1$ and $x_2$ are perfect matchings of $Q$.

Recall that a simple $A$-module $V$ of dimension $1^{Q_0}$ is characterized by the property that there is a cycle which passes through each vertex of $Q$ that does not annihilate $V$.
Clearly $Q \setminus x_1$ and $Q \setminus x_2$ each contain a cycle that passes through each vertex of $Q$.
Therefore $x_1$ and $x_2$ are simple matchings.
\end{proof}

\begin{Lemma} \label{at least one}
If $A$ is cancellative, then $Q$ has at least one simple matching.
\end{Lemma}

\begin{proof}
Follows from Proposition \ref{here} and Lemma \ref{figure3lemma}.
\end{proof}

Lemma \ref{at least one} will be superseded by Theorem \ref{AKZ} below.

\begin{Lemma} \label{here3}
Let $u \in \mathbb{Z}^2 \setminus 0$, and suppose $\hat{\mathcal{C}}^u_i \not = \emptyset$ for each $i \in Q_0$.
Then $A$ does not have a non-cancellative pair where one of the paths is a vertex.
\end{Lemma}

\begin{proof}
Suppose $e_i, p$ is a non-cancellative pair.
Then there is some $m,n \geq 0$ such that
$$p\sigma_i^m = e_i \sigma_i^n = \sigma_i^n,$$
by Lemma \ref{here2}.1.
Whence $\overbar{p} = \sigma^{n-m}$ since $B$ is an integral domain.
But $n - m > 0$ since $p \not = e_i$.
Furthermore,
$$\sigma^{n-m} = \overbar{p} = \overbar{e}_i = 1,$$
by Lemma \ref{here2}.3.
Thus $\sigma = 1$.
Therefore $Q$ has no simple matchings, $\mathcal{S} = \emptyset$.
Consequently, for each $u \in \mathbb{Z}^2 \setminus 0$ there is a vertex $i \in Q_0$ such that $\hat{\mathcal{C}}^u_i = \emptyset$, by Lemma \ref{figure3lemma}.
\end{proof}

\begin{Lemma} \label{longlist}
Let $u \in \mathbb{Z}^2$.
If $p,q \in \mathcal{C}^u$, then there is some $n \in \mathbb{Z}$ such that $\overbar{p} = \overbar{q}\sigma^n$.
In particular, if $\sigma \nmid \overbar{p}$ and $\sigma \nmid \overbar{q}$, then $\overbar{p} = \overbar{q}$.
\end{Lemma}

\begin{proof}
Consider cycles $p,q \in \mathcal{C}^u$.
Since $Q$ is a dimer quiver, there is a path $r$ from $\operatorname{t}(p)$ to $\operatorname{t}(q)$.
Thus there is some $m, n \geq 0$ such that
$$rp \sigma_{\operatorname{t}(p)}^m = qr \sigma_{\operatorname{t}(p)}^n,$$
by Lemma \ref{here2}.1.
Furthermore, $\tau$ is an algebra homomorphism by Lemma \ref{tau'A'}.
Thus
$$\overbar{r}  \overbar{p} \sigma^m = \bar{\tau}\left(rp \sigma_{\operatorname{t}(p)}^n \right) = \bar{\tau}\left(qr \sigma_{\operatorname{t}(p)}^n \right) = \overbar{q}  \overbar{r} \sigma^n.$$
Therefore $\overbar{p} = \overbar{q} \sigma^{n-m}$ since $B$ is an integral domain.
\end{proof}

\begin{Lemma} \label{longlist3}
Suppose $A$ is cancellative.
\begin{enumerate}
 \item If $a \in Q_1$, $p \in \hat{\mathcal{C}}_{\operatorname{t}(a)}^u$, and $q \in \hat{\mathcal{C}}_{\operatorname{h}(a)}^u$, then $ap = qa$.
 \item Each vertex corner ring $e_iAe_i$ is commutative.
\end{enumerate}
\end{Lemma}

\begin{proof}
(1) Suppose $a \in Q_1$, $p \in \hat{\mathcal{C}}_{\operatorname{t}(a)}^u$, and $q \in \hat{\mathcal{C}}_{\operatorname{h}(a)}^u$.
Then
$$\operatorname{t}\left((ap)^+\right) = \operatorname{t}\left((qa)^+\right) \ \ \text{ and } \ \ \operatorname{h}\left((ap)^+\right) = \operatorname{h}\left((qa)^+\right).$$
Let $r^+$ be a path in $Q^+$ from $\operatorname{h}\left( (ap)^+ \right)$ to $\operatorname{t}\left( (ap)^+ \right)$.
Then by Lemma \ref{here3'}.2, there is some $m, n \geq 1$ such that
$$rap = \sigma_{\operatorname{t}(a)}^m \ \ \text{ and } \ \ rqa = \sigma_{\operatorname{t}(a)}^n.$$

Assume to the contrary that $m < n$.
Then $qa = ap \sigma_{\operatorname{t}(a)}^{n-m}$ since $A$ is cancellative.
Let $b$ be a path such that $ab$ is a unit cycle.
By Lemma \ref{sigma}, we have
$$q \sigma_{\operatorname{h}(a)} = qab = ap \sigma_{\operatorname{t}(a)}^{n-m} b = apb \sigma^{n-m}_{\operatorname{h}(a)}.$$
Thus, since $A$ is cancellative and $n-m \geq 1$,
$$q = apb \sigma^{n-m-1}_{\operatorname{h}(a)}.$$
Furthermore, $(ba)^+$ is cycle in $Q^+$ since $ba$ is a unit cycle.
But this is a contradiction since $q$ is in $\hat{\mathcal{C}}^u$.
Whence $m = n$.
Therefore $qa = ap$.

(2) Consider cycles $p,q \in e_iAe_i$.
Since $I$ is generated by binomials, it suffices to show that $qp = pq$.
Let $r^+$ be a path in $Q^+$ from $\operatorname{h}\left((qp)^+\right)$ to $\operatorname{t}\left((qp)^+\right)$.
Set $r := \pi(r^+)$.
Then $(rqp)^+$ and $(rpq)^+$ are cycles.
In particular, there is some $m,n \geq 1$ such that
$$rqp = \sigma_i^m \ \ \text{ and } \ \ rpq = \sigma_i^n,$$
by Lemma \ref{here3'}.2.
Thus, since $\tau$ is an algebra homomorphism,
$$\sigma^m = \overbar{rqp} = \overbar{r} \overbar{q} \overbar{p} = \overbar{r} \overbar{p} \overbar{q} = \overbar{rpq} = \sigma^n.$$
Furthermore, $\sigma \not = 1$ by Lemma \ref{at least one}.
Whence $m = n$ since $B$ is an integral domain.
Thus $rqp = \sigma_i^m = rpq$.
Therefore $qp = pq$ since $A$ is cancellative.
\end{proof}

\begin{Proposition} \label{circle 1}
Let $u \in \mathbb{Z}^2$.
Suppose (i) $A$ is cancellative, or (ii) $\hat{\mathcal{C}}_i^u \not = \emptyset$ for each $i \in Q_0$.
\begin{enumerate}
 \item If $p \in \hat{\mathcal{C}}^u$, then $\sigma \nmid \overbar{p}$.
 \item If $p,q \in \hat{\mathcal{C}}^u$, then $\overbar{p} = \overbar{q}$.
 \item If a cycle $p$ is formed from subpaths of cycles in $\hat{\mathcal{C}}^u$, then $p \in \hat{\mathcal{C}}$.
 \item If $p,q \in \hat{\mathcal{C}}_i^u$, then $p = q$.
\end{enumerate}
\end{Proposition}

Note that in Claim (4) the cycles $p$ and $q$ are based at the same vertex $i$, whereas in Claim (2) $p$ and $q$ may be based at different vertices.

\begin{proof}
If $A$ is cancellative, then $\hat{\mathcal{C}}_i^u \not = \emptyset$ for each $i \in Q_0$, by Proposition \ref{here}.
Therefore assumption (i) implies assumption (ii), and so it suffices to suppose (ii) holds.

(1) The set $\hat{\mathcal{C}}^0$ consists of the vertices of $Q$. 
Whence $p \in \hat{\mathcal{C}}^0$ implies $\sigma \nmid \overbar{p}$.
So suppose $u \in \mathbb{Z}^2 \setminus 0$.

Fix a maximal set $\widehat{\mathcal{C}}^u$ of representatives of cycles in $\hat{\mathcal{C}}^u$ whose lifts do not intersect transversely in $\mathbb{R}^2$.
Consider a cycle $p \in \hat{\mathcal{C}}^u$.
If $p$ has a representative $\tilde{p}$ in $\widehat{\mathcal{C}}^u$, then $\sigma \nmid \overbar{p}$, by Lemma \ref{figure3lemma}.
It thus suffices to suppose that $p$ does not have a representative belonging to $\widehat{\mathcal{C}}^u$.
In particular, for any representative $\tilde{p}$ of $p$, there are cycles $\tilde{s}, \tilde{t} \in \widehat{\mathcal{C}}^u$ such that
$$s = p_3q_2p_1, \ \ \ t = q_3p_2q_1, \ \ \ p = p_3p_2p_1,$$
as in Figure \ref{figure2}.

Since $\tilde{s},\tilde{t} \in \widehat{\mathcal{C}}^u$, we have $\sigma \nmid \overbar{s}$ and $\sigma \nmid \overbar{t}$.
Thus
\begin{equation} \label{no black holes}
\overbar{s} = \overbar{t}
\end{equation}
by Lemma \ref{longlist}.
Set $q := q_3q_2q_1$.

Assume to the contrary that $\sigma \mid \overbar{p}$.
Then
$$\sigma \mid \overbar{p} \overbar{q} = \overbar{p}_3\overbar{p}_2\overbar{p}_1 \overbar{q}_3\overbar{q}_2\overbar{q}_1 = \overbar{s} \overbar{t} \stackrel{\textsc{(i)}}{=} \overbar{s}^2,$$
where (\textsc{i}) holds by (\ref{no black holes}).
Therefore $\sigma \mid \overbar{s}$ since $\sigma = \prod_{x \in \mathcal{S}}x$.
But this is not possible since $\tilde{s} \in \widehat{\mathcal{C}}^u$.

(2) Follows from Claim (1) and Lemma \ref{longlist}.

A direct proof assuming $A$ is cancellative:
Suppose $p,q \in \hat{\mathcal{C}}^u$.
Let $r$ be a path from ${\operatorname{t}(p)}$ to ${\operatorname{t}(q)}$.
Then $rp = qr$ by Lemma \ref{longlist3}.1.
Thus
$$\overbar{r}  \overbar{p} = \overbar{rp} = \overbar{qr} = \overbar{q}  \overbar{r} = \overbar{r} \overbar{q}.$$
Therefore $\overbar{p} = \overbar{q}$ since $B$ is an integral domain.

(3) Let $p = p_{\ell} \cdots p_1 \in \mathcal{C}^u$ be a cycle formed from subpaths $p_j$ of cycles $q_j$ in $\hat{\mathcal{C}}^u$.
Then
$$g := \overbar{q}_1 = \cdots = \overbar{q}_{\ell}$$
by Claim (2).
Furthermore, $\sigma \nmid g$ by Claim (1).
In particular, there is a simple matching $x \in \mathcal{S}$ for which $x \nmid g$.
Whence $x \nmid \overbar{p}_j$ for each $1 \leq j \leq \ell$.
Thus $x \nmid \overbar{p}$.
Therefore $\sigma \nmid \overbar{p}$.
Consequently, $p \in \hat{\mathcal{C}}$ by the contrapositive of Lemma \ref{here3'}.3, with Lemma \ref{here3}.

(4) Follows from Claim (3) and Figure \ref{figure3}.ii.

A direct proof assuming $A$ is cancellative:
Suppose $p,q \in \hat{\mathcal{C}}_i^u$.
Let $r^+$ be a path in $Q^+$ from $\operatorname{h}\left(p^+ \right)$ to $\operatorname{t}\left( p^+ \right)$.
Then there is some $m,n \geq 1$ such that
$$rp = \sigma_i^m \ \ \text{ and } \ \ rq = \sigma_i^n,$$
by Lemma \ref{here3'}.2.
Suppose $m \leq n$.
Then $rp \sigma_i^{n-m} = \sigma_i^n = rq$.
Thus $p\sigma_i^{n-m} = q$ since $A$ is cancellative.
But then Claim (1) implies $m = n$ since by assumption $p$ and $q$ are in $\hat{\mathcal{C}}^u$.
Therefore $p = q$.
\end{proof}

\begin{Lemma} \label{area to zero}
Suppose $\hat{\mathcal{C}}^u_i \not = \emptyset$ for each $u \in \mathbb{Z}^2$ and $i \in Q_0$.
Let $\varepsilon_1, \varepsilon_2 \in \mathbb{Z}$.
There is an $n \gg 1$ such that if
$$p \in \hat{\mathcal{C}}^{(\varepsilon_1,0)}_i \ \ \ \text{ and } \ \ \ q \in \hat{\mathcal{C}}^{(n \varepsilon_1, \varepsilon_2)}_i,$$
then $\sigma \nmid \overbar{pq}$.
\end{Lemma}

\begin{proof}
Fix $\varepsilon_1, \varepsilon_2 \in \mathbb{Z}$ and $i \in Q_0$.
Consider the set of cycles
$$p \in \hat{\mathcal{C}}^{(\varepsilon_1,0)}_i \ \ \ \text{ and } \ \ \  q_n \in \hat{\mathcal{C}}^{(n \varepsilon_1, \varepsilon_2)}_i,$$
with $n \geq 0$.

(i) We claim that for each $n \geq 1$, $q_n^+$ lies in the region $\mathcal{R}_{p^2q_{n-1},q_{n+1}}$ (modulo $I$).
This is shown in Figure \ref{area to zero fig}.i.
Indeed, suppose a representative $\tilde{q}_n^+$ of $q_n^+$ intersects a representative $\tilde{q}_{n-1}^+$ of $q_{n-1}^+$, as shown in Figure \ref{area to zero fig}.ii.
Then $q_{n-1}$ and $q_n$ factor into paths
$$q_{n-1} = s_3s_2s_1 \ \ \ \text{ and } \ \ \ q_n = t_3t_2t_1,$$
where $\operatorname{t}(s_2^+) = \operatorname{t}(t_2^+)$ and $\operatorname{h}(s_2^+) = \operatorname{h}(t_2^+)$.
In particular, there is some $m \in \mathbb{Z}$ such that
$$\overbar{s}_2 = \overbar{t}_2 \sigma^m,$$
by Lemma \ref{here2}.2.
Set $r := t_3s_2t_1$.

Since $q_{n-1}$ and $q_n$ are in $\hat{\mathcal{C}}$, we have
\begin{equation} \label{sigma not divide}
\sigma \nmid \overbar{q}_{n-1} \ \ \ \text{ and } \ \ \ \sigma \nmid \overbar{q}_n,
\end{equation}
by Proposition \ref{circle 1}.1.
Hence $\sigma \nmid \overbar{s}_2$ and $\sigma \nmid \overbar{t}_2$.
Thus $m = 0$.
Therefore
$$\overbar{r} = \overbar{t}_3 \overbar{s}_2 \overbar{t}_1 = \overbar{t}_3\overbar{t}_2\overbar{t}_1 = \overbar{q}_n.$$
In particular, $\sigma \nmid \overbar{r}$ by (\ref{sigma not divide}).
Thus the cycle $r$ is in $\hat{\mathcal{C}}$ by Lemmas \ref{here3'}.4 and \ref{here3}.
Furthermore, $r$ is in $\mathcal{C}^{(n \varepsilon_1, \varepsilon_2)}_i$ by construction.
Whence $r$ is in $\hat{\mathcal{C}}^{(n \varepsilon_1, \varepsilon_2)}_i$.
Therefore $r = q_n$ (modulo $I$) by Proposition \ref{circle 1}.4.
This proves our claim.

(ii) By Claim (i), there is a cycle $s \in \hat{\mathcal{C}}^{(\varepsilon_1,0)}$ such that for each $n \geq 1$, the area of the region
$$\mathcal{R}_{sq'_{n},q'_{n+1}},$$
bounded by a rightmost subpath ${q'}^{+}_{n}$ of $q_{n}^+$, a rightmost subpath ${q'}^{+}_{n+1}$ of $q_{n+1}^+$, and $s^+$, tends to zero (modulo $I$) as $n \to \infty$.
See Figure \ref{area to zero fig}.iii.
(The case $r = p$ is shown in Figure \ref{area to zero fig}.i.)
Since $Q$ is finite, there is some $N \gg 1$ such that if $n \geq N$, then
\begin{equation*} \label{snq'n}
q'_{n+1} = sq'_{n} \ \ \ \text{(modulo $I$)}.
\end{equation*}

Fix $n \geq N$.
There is a simple matching $x \in \mathcal{S}$ such that $x \nmid \overbar{q}_{n}$, by Proposition \ref{circle 1}.1.
Whence
$$x \nmid \overbar{q}'_{n+1} = \overbar{s} \overbar{q}'_{n}.$$
In particular, $x \nmid \overbar{s}$.
But $\overbar{p} = \overbar{s}$ since $p$ and $s$ are both in $\hat{\mathcal{C}}^{(\varepsilon_1,0)}$, by Proposition \ref{circle 1}.2.
Thus $x \nmid \overbar{p}$.
Therefore $x \nmid \overbar{p} \overbar{q}_{n}$, and so $\sigma \nmid \overbar{p}\overbar{q}_{n}$.
\end{proof}

\begin{figure}
$$\begin{array}{cc}
(i) &
\xy 0;/r.3pc/:
(-30,-10)*{\cdot}="1";(-20,-10)*{\cdot}="2";(-10,-10)*{\cdot}="3";(0,-10)*{\cdot}="4";(10,-10)*{\cdot}="5";(20,-10)*{\cdot}="6";(30,-10)*{\cdot}="7";(40,-10)*{\cdot}="8";
(-30,10)*{\cdot}="9";(-20,10)*{\cdot}="10";(-10,10)*{\cdot}="11";(0,10)*{\cdot}="12";(10,10)*{\cdot}="13";(20,10)*{\cdot}="14";(30,10)*{\cdot}="15";(40,10)*{\cdot}="16";
(-30,11.7)*{\text{\scriptsize{$q_0$}}}="";(-20,11.7)*{\text{\scriptsize{$q_1$}}}="";(-10,11.7)*{\text{\scriptsize{$q_2$}}}="";(0,11.7)*{\text{\scriptsize{$q_3$}}}="";
(15,11.7)*{\cdots}="";
(30,11.7)*{\text{\scriptsize{$q_n$}}}="";(40,11.7)*{\text{\scriptsize{$q_{n+1}$}}}="";
{\ar"1";"9"};
{\ar@[red]_p"2";"3"};{\ar@[red]_p"3";"4"};{\ar@[red]"4";"5"};{\ar@[red]"5";"6"};{\ar@[red]"6";"7"};{\ar@[red]"7";"8"};
{\ar"1";"10"};{\ar"1";"11"};{\ar"1";"12"};{\ar"1";"13"};{\ar"1";"14"};{\ar@[blue]"1";"15"};{\ar@[blue]"1";"16"};
{\ar@[red]^{p}"15";"16"};{\ar@[red]_{p}"1";"2"};
{\ar@[red]^p"9";"10"};{\ar@[red]^p"10";"11"};{\ar@[red]^p"11";"12"};{\ar@[red]"12";"13"};{\ar@[red]"13";"14"};{\ar@[red]"14";"15"};
\endxy
\\ \\
(ii) &
\xy 0;/r.52pc/:
(-15,-5)*{\cdot}="1";(-5,-5)*{\cdot}="2";(5,-5)*{\cdot}="3";(15,-5)*{\cdot}="4";
(-15,5)*{}="5";(-5,5)*{\cdot}="6";(5,5)*{\cdot}="7";(15,5)*{\cdot}="8";
(-10,-2.5)*{}="9";(-10,0)*{\cdot}="10";(-7.5,2.5)*{\cdot}="11";(0,2.5)*{}="12";
(-5,5.9)*{\text{\scriptsize{$q_{n-1}$}}}="";(5,5.9)*{\text{\scriptsize{$q_n$}}}="";(15,5.9)*{\text{\scriptsize{$q_{n+1}$}}}="";
{\ar@[red]_p"1";"2"};{\ar_p"2";"3"};{\ar_p"3";"4"};
{\ar^p@[red]"6";"7"};{\ar^p"7";"8"};
{\ar"1";"8"};
{\ar@[blue]@{-}"1";"9"};{\ar@[blue]_{t_1}"9";"10"};{\ar@/^.6pc/^{t_2}@[blue]"10";"11"};{\ar@[blue]@{-}_{t_3}"11";"12"};{\ar@[blue]"12";"7"};
{\ar@[purple]^{s_1}"1";"10"};{\ar@[purple]_{s_2}"10";"11"};{\ar@[purple]^{s_3}"11";"6"};
\endxy
\\ \\
(iii) &
\xy 0;/r.3pc/:
(-30,-10)*{\cdot}="1";(-20,-10)*{\cdot}="2";(-10,-10)*{\cdot}="3";(0,-10)*{\cdot}="4";(10,-10)*{\cdot}="5";(20,-10)*{\cdot}="6";(30,-10)*{\cdot}="7";(40,-10)*{\cdot}="8";
(-30,10)*{\cdot}="9";
(-20,2.5)*{\cdot}="10";(-10,2.5)*{\cdot}="11";(0,2.5)*{\cdot}="12";(10,2.5)*{\cdot}="13";(20,2.5)*{\cdot}="14";(30,2.5)*{\cdot}="15";(40,2.5)*{\cdot}="16";
(-20,10)*{\cdot}="10'";(-10,10)*{\cdot}="11'";(0,10)*{\cdot}="12'";(10,10)*{\cdot}="13'";(20,10)*{\cdot}="14'";(30,10)*{\cdot}="15'";(40,10)*{\cdot}="16'";
(-30,11.7)*{\text{\scriptsize{$q_0$}}}="";(-20,11.7)*{\text{\scriptsize{$q_1$}}}="";(-10,11.7)*{\text{\scriptsize{$q_2$}}}="";(0,11.7)*{\text{\scriptsize{$q_3$}}}="";
(15,11.7)*{\cdots}="";
(30,11.7)*{\text{\scriptsize{$q_n$}}}="";(40,11.7)*{\text{\scriptsize{$q_{n+1}$}}}="";
(-30,2.5)*{\cdot}="25";
{\ar"1";"25"};{\ar"25";"9"};
{\ar@[red]_p"2";"3"};{\ar@[red]_p"3";"4"};{\ar@[red]"4";"5"};{\ar@[red]"5";"6"};{\ar@[red]"6";"7"};{\ar@[red]"7";"8"};
{\ar"1";"10"};{\ar"1";"11"};{\ar"1";"12"};{\ar"1";"13"};{\ar"1";"14"};{\ar@[blue]"1";"15"};{\ar@[blue]"1";"16"};
{\ar@[red]^{p}"15'";"16'"};{\ar@[red]_{p}"1";"2"};
{\ar"10";"10'"};{\ar"11";"11'"};{\ar"12";"12'"};{\ar"13";"13'"};{\ar"14";"14'"};{\ar@[blue]"15";"15'"};{\ar@[blue]"16";"16'"};
{\ar@[orange]^{s}"10";"11"};{\ar@[orange]^{s}"11";"12"};{\ar@[orange]"12";"13"};{\ar@[orange]"13";"14"};{\ar@[orange]"14";"15"};{\ar@[orange]^{s}"15";"16"};
{\ar@[red]^p"9";"10'"};{\ar@[red]^p"10'";"11'"};{\ar@[red]^p"11'";"12'"};{\ar@[red]"12'";"13'"};{\ar@[red]"13'";"14'"};{\ar@[red]"14'";"15'"};
{\ar@[orange]^{s}"25";"10"};
\endxy
\end{array}$$
\caption{Setups for Lemma \ref{area to zero}, drawn on the cover $Q^+$.
(In each case, $q_j$ labels the path from the lower line of $p$'s to the upper line of $p$'s with head at the corresponding vertex.)
In (ii): $q_{n-1}^+$ and $q_n^+$ intersect, and thus factor into paths $q_{n-1} = s_3s_2s_1$ and $q_n = t_3t_2t_1$, drawn in purple and blue respectively.
In (iii): the area of the region $\mathcal{R}_{s_nq'_{n},q'_{n+1}}$ tends to zero as $n \to \infty$.
Here, an infinite path in $\pi^{-1}(r^{\infty})$ is drawn in orange, each lift of the cycle $p$ is drawn in red, and $q^+_n$ and $q^+_{n+1}$ are drawn in blue.}
\label{area to zero fig}
\end{figure}

Consider the subset of arrows
\begin{equation*} \label{Q dagger def}
Q_1^{\mathcal{S}} := \left\{ a \in Q_1 \ | \ a \not \in x \text{ for each } x \in \mathcal{S} \right\},
\end{equation*}
where $\mathcal{S}$ is the set of simple matchings of $A$.

We will show in Theorem \ref{AKZ} below that the two assumptions considered in the following lemma, namely that $Q_1^{\mathcal{S}} \not = \emptyset$ and $\hat{\mathcal{C}}^u_i \not = \emptyset$ for each $u \in \mathbb{Z}^2$ and $i \in Q_0$, can never both hold.

\begin{Lemma} \label{purple lemma}
Suppose $\hat{\mathcal{C}}^u_i \not = \emptyset$ for each $u \in \mathbb{Z}^2$ and $i \in Q_0$.
Let $\delta \in Q_1^{\mathcal{S}}$.
There is a cycle $p \in \hat{\mathcal{C}}_{\operatorname{t}(\delta)}$ such that $\delta$ is not the rightmost arrow subpath of any representative of $p$.
\end{Lemma}

\begin{proof}
Assume to the contrary that there is an arrow $\delta$ which is a rightmost arrow subpath of some representative of each cycle in $\hat{\mathcal{C}}_{\operatorname{t}(\delta)}$.

(i) We first claim that there is some $u \in \mathbb{Z}^2 \setminus 0$ such that $\operatorname{t}(\delta^+)$ lies in the interior of the region $\mathcal{R}_{\tilde{s},\tilde{t}}$ bounded by the lifts of two representatives $\tilde{s},\tilde{t}$ of the (unique) cycle in $\hat{\mathcal{C}}^u_{\operatorname{h}(\delta)}$.
(There is precisely one cycle in $\hat{\mathcal{C}}^u_{\operatorname{h}(\delta)}$ by Proposition \ref{circle 1}.4.)

Indeed, by Lemma \ref{area to zero}, there are counter-clockwise ordered vectors
$$u_0 = u_{n+1}, u_1, u_2, \ldots, u_n \in \mathbb{Z}^2 \setminus 0$$
whose convex cone is $\mathbb{R}^2$, and cycles 
$$p_m = \delta p'_m \in \hat{\mathcal{C}}^{u_m}_{\operatorname{h}(\delta)},$$
such that for $0 \leq m \leq n$, we have
\begin{equation} \label{m+1 m}
\sigma \nmid \overbar{p_{m+1}p_m}.
\end{equation}

Assume to the contrary that for each $m$, $\delta^+$ is not contained in the region
$$\mathcal{R}_{p_{m+1}p_m,p_mp_{m+1}}.$$
See Figure \ref{purple fig2}.i. 
Since the convex cone of $u_0, \ldots, u_n$ is $\mathbb{R}^2$, there is some $m$ for which $p_m$ contains a leftmost subpath $\delta r$, and $p_{m+1}$ contains a rightmost subpath $r'$, such that $(\delta rr')^+$ is a cycle in $Q^+$.
See Figure \ref{uf}. 
Whence,
$$\overbar{\delta r r'} = \sigma^{\ell}$$
for some $\ell \geq 1$, by Lemmas \ref{here3'}.1 and \ref{here3}.
Consequently, $\sigma \mid \overbar{p_{m+1}p_m}$.
But this is a contradiction to (\ref{m+1 m}). 

Thus there is some $0 \leq m \leq n$ such that $\delta^+$ lies in $\mathcal{R}_{p_{m+1}p_m,p_mp_{m+1}}$; see Figure \ref{purple fig2}.ii.
Since $p_m$, $p_{m+1}$ are in $\hat{\mathcal{C}}$, their lifts $p_m^+$, $p_{m+1}^+$ do not have cyclic subpaths.
Thus, $\delta^+$ only meets $p_m^+$ and $p_{m+1}^+$ at $\operatorname{t}(p_m^+) = \operatorname{h}(\delta^+)$.
Consequently, $\operatorname{t}(\delta^+)$ lies in the interior of $\mathcal{R}_{p_{m+1}p_m,p_mp_{m+1}}$.
The claim then follows by setting
$$s = p_{m+1}p_m \ \ \ \text{ and } \ \ \ t = p_m p_{m+1}.$$

(ii) Let $s$ and $t$ be as in Claim (i).
In particular, $\overbar{s} = \overbar{t}$.
Assume to the contrary that there is a simple matching $x \in \mathcal{S}$ for which
$$x \nmid \overbar{s} = \overbar{t}.$$
Let $r^+$ be a path in $\mathcal{R}_{\tilde{s},\tilde{t}}$ from a vertex on the boundary of $\mathcal{R}_{\tilde{s},\tilde{t}}$ to $\operatorname{t}(\delta^+)$.
It suffices to suppose the tail of $r$ is a trivial subpath of $\tilde{s}$, in which case $\tilde{s}$ factors into paths
$$\tilde{s} = \tilde{s}_2e_{\operatorname{t}(r)}\tilde{s}_1.$$
See Figure \ref{purple fig2}.ii.
Then $(rs_1\delta)^+$ is a cycle in $Q^+$.
Whence
$$\overbar{rs_1\delta} = \sigma^{\ell}$$
for some $\ell \geq 1$, by Lemmas \ref{here3'}.1 and \ref{here3}.
Consequently,
$$x \mid \overbar{rs_1\delta} = \overbar{r}\overbar{s}_1\overbar{\delta}.$$
Furthermore, by assumption,
$$x \nmid \overbar{s}_1 \ \ \ \text{ and } \ \ \ x \nmid \overbar{\delta}.$$
Whence $x \mid \overbar{r}$.
Thus, since $r$ is arbitrary, $x$ divides the $\bar{\tau}$-image of each path in $\mathcal{R}_{\tilde{s},\tilde{t}}$ from the boundary of $\mathcal{R}_{\tilde{s},\tilde{t}}$ to $\operatorname{t}(\delta^+)$.
But $\operatorname{t}(\delta^+)$ lies in the interior of $\mathcal{R}_{\tilde{s},\tilde{t}}$.
Thus the vertex $\operatorname{t}(\delta)$ is a source in $Q \setminus x$.
Therefore $x$ is not simple, contrary to assumption.
It follows that $x \mid \overbar{s}$ for each $x \in \mathcal{S}$.

(iii) By Claim (ii), $\sigma \mid \overbar{s}$.
But this is a contradiction since $s$ is in $\hat{\mathcal{C}}$, by Lemmas \ref{here3'}.3 and \ref{here3}.
Therefore there is a cycle $p$ in $\hat{\mathcal{C}}_{\operatorname{t}(\delta)}$ such that $\delta$ is not the rightmost arrow subpath of any representative of $p$.
\end{proof}

\begin{figure}
$$\begin{array}{ccc}
\xy
0;/r.3pc/:
(-20,0)*{\cdot}="1";(20,0)*{\cdot}="2";(-11,8)*{}="3";(11,8)*{}="4";
(-11,-8)*{}="5";(11,-8)*{}="6";
(-29,0)*{\cdot}="8";
(29,0)*{\cdot}="9";{\ar^{\delta}"2";"9"};
{\ar@{-}@/^/@[blue]"1";"3"};{\ar@{-}^{p_mp_{m+1}}@[blue]"3";"4"};{\ar@/^/@[blue]"4";"2"};
{\ar@{-}@/_/@[red]"1";"5"};{\ar@{-}_{p_{m+1}p_m}@[red]"5";"6"};{\ar@/_/@[red]"6";"2"};
{\ar^{\delta}@[green]"8";"1"};
\endxy
& \ \ & 
\xy
0;/r.3pc/:
(-20,0)*{\cdot}="1";(20,0)*{\cdot}="2";(-11,8)*{}="3";(11,8)*{}="4";
(-11,-8)*{\cdot}="5";(11,-8)*{}="6";
(-11,0)*{\cdot}="8";
(11,0)*{\cdot}="9";{\ar_{\delta}"2";"9"};
{\ar@{-}@/^/@[blue]"1";"3"};{\ar@{-}^{t = p_mp_{m+1}}@[blue]"3";"4"};{\ar@/^/@[blue]"4";"2"};
{\ar@/_/^{s_1}@[red]"1";"5"};{\ar@{-}_{s = p_{m+1}p_m}@[red]"5";"6"};{\ar@/_/^{s_2}@[red]"6";"2"};
{\ar_{r}"5";"8"};{\ar_{\delta}@[green]"8";"1"};
\endxy
\\
(i) & & (ii) 
\end{array}$$
\caption{Setups for Lemma \ref{purple lemma}, drawn on the cover $Q^+$.
In both cases, $\delta^+$ is the green $\delta$ arrow on the left.
The black $\delta$ arrow on the right is the last arrow subpath of both $p_mp_{m+1}$ and $p_{m+1}p_m$.
In (i), $\delta^+$ is not contained in the region $\mathcal{R}_{p_{m+1}p_m,p_mp_{m+1}}$, whereas in (ii) $\delta^+$ is contained in this region.}
\label{purple fig2}
\end{figure}

\begin{figure}
$$\xy
0;/r.3pc/:
(-24,-3)*{}="1";(-6,-3)*{\cdot}="2";(0,-3)*{\cdot}="3";
(0,3)*{\cdot}="4";(6,3)*{\cdot}="5";(24,3)*{}="6";
(0,-21)*{}="7";(0,21)*{}="8";
(24,21)*{}="9";(-24,-21)*{}="10";
{\ar@[brown]^{p_{m-1}}"7";"3"};{\ar@[brown]^{p_{m-1}}"4";"8"};
{\ar@[blue]^{p_m}"9";"5"};{\ar@[blue]^{p_m}"9";"5"};{\ar@[blue]^r"5";"3"};{\ar@[blue]"4";"2"};{\ar@[blue]^{p_m}"2";"10"};{\ar@[red]^{p_{m+1}}"1";"2"};{\ar@[red]^{r'}"4";"5"};{\ar@[red]_{p_{m+1}}"5";"6"};
{\ar@[green]|-{\delta}"3";"4"};{\ar@[red]"2";"3"};
\endxy$$
\caption{Setup for Lemma \ref{purple lemma}, Claim (i).}
\label{uf}
\end{figure}

\begin{Remark} \rm{
There are dimer algebras that have an arrow $\delta \in Q_1$ which is a rightmost arrow subpath of each cycle in $\hat{\mathcal{C}}_{\operatorname{t}(\delta)}$; see Figure \ref{AKZfig1}.
Furthermore, if $A$ has center $Z$, admits a cyclic contraction, and $\delta \in Q_1^{\mathcal{S}}$, then $\delta$ is a rightmost arrow subpath of each cycle $p \in Ze_{\operatorname{t}(\delta)}$ for which $\sigma \nmid \overbar{p}$, by \cite[Lemma 2.4]{B2}.
} \end{Remark}

\begin{figure}
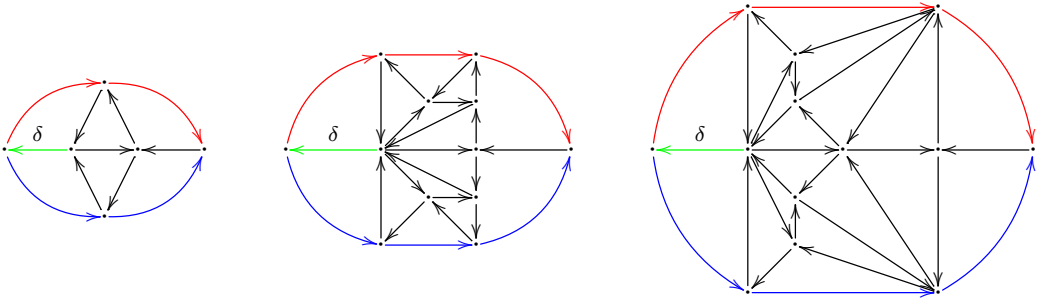

$$\xy 0;/r.3pc/:
(-10.5,0)*{\cdot}="1";(-3.5,0)*{\cdot}="2";(3.5,0)*{\cdot}="3";(10.5,0)*{\cdot}="4";
(0,-7)*{\cdot}="5";(0,7)*{\cdot}="6";
{\ar@[green]_{\delta}"2";"1"};{\ar"2";"3"};{\ar"4";"3"};{\ar"6";"2"};{\ar"5";"2"};
{\ar"3";"5"};{\ar"3";"6"};
{\ar@[blue]@/_.7pc/"1";"5"};{\ar@[red]@/^.7pc/"1";"6"};
{\ar@[blue]@/_.7pc/"5";"4"};{\ar@[red]@/^.7pc/"6";"4"};
\endxy
\ \ \ \ \ \ \
\xy 0;/r.3pc/:
(-15,0)*{\cdot}="1";(-5,0)*{\cdot}="2";(5,0)*{\cdot}="3";(15,0)*{\cdot}="4";
(-5,-10)*{\cdot}="5";(5,-10)*{\cdot}="6";
(0,-5)*{\cdot}="7";(0,5)*{\cdot}="8";
(-5,10)*{\cdot}="9";(5,10)*{\cdot}="10";
(5,-5)*{\cdot}="11";(5,5)*{\cdot}="12";
{\ar@[blue]@/_.7pc/"1";"5"};{\ar@[red]@/^.7pc/"1";"9"};
{\ar@[green]_{\delta}"2";"1"};
{\ar@[blue]"5";"6"};{\ar@[red]"9";"10"};
{\ar@[blue]@/_.7pc/"6";"4"};{\ar@[red]@/^.7pc/"10";"4"};
{\ar"2";"3"};{\ar"4";"3"};
{\ar"9";"2"};{\ar"5";"2"};{\ar"8";"9"};{\ar"7";"5"};{\ar"2";"8"};{\ar"2";"7"};
{\ar"10";"8"};{\ar"8";"12"};{\ar"12";"2"};
{\ar"6";"7"};{\ar"7";"11"};{\ar"11";"2"};
{\ar"3";"12"};{\ar"12";"10"};{\ar"3";"11"};{\ar"11";"6"};
\endxy
\ \ \ \ \ \ \
\xy 0;/r.3pc/:
(-20,0)*{\cdot}="1";(-10,-15)*{\cdot}="2";(10,-15)*{\cdot}="3";(20,0)*{\cdot}="4";(10,15)*{\cdot}="5";(-10,15)*{\cdot}="6";
(-10,0)*{\cdot}="7";(0,0)*{\cdot}="8";(10,0)*{\cdot}="9";
(-5,-10)*{\cdot}="10";(-5,10)*{\cdot}="11";
(-5,-5)*{\cdot}="12";(-5,5)*{\cdot}="13";
{\ar@[green]_{\delta}"7";"1"};
{\ar@/_.7pc/@[blue]"1";"2"};{\ar@[blue]"2";"3"};{\ar@/_.7pc/@[blue]"3";"4"};
{\ar@/^.7pc/@[red]"1";"6"};{\ar@[red]"6";"5"};{\ar@[red]@/^.7pc/"5";"4"};
{\ar"6";"7"};{\ar"2";"7"};
{\ar"10";"2"};{\ar"11";"6"};
{\ar"5";"11"};{\ar"3";"10"};
{\ar"10";"12"};{\ar"11";"13"};
{\ar"13";"7"};{\ar"12";"7"};
{\ar"8";"13"};{\ar"8";"12"};
{\ar"5";"8"};{\ar"3";"8"};
{\ar"9";"5"};{\ar"9";"3"};
{\ar"4";"9"};{\ar"8";"9"};{\ar"7";"8"};
{\ar"7";"10"};{\ar"7";"11"};
{\ar"12";"3"};{\ar"13";"5"};
\endxy$$
\caption{In each example, the arrow $\delta$ is a rightmost arrow subpath of each path from $\operatorname{t}(\delta)$ to a vertex on the boundary of the region $\mathcal{R}_{p,q}$ bounded by the paths $p^+$ and $q^+$, drawn in red and blue respectively.
The red and blue arrows are nontrivial paths in $Q^+$, and the black arrows are arrows in $Q^+$.}
\label{AKZfig1}
\end{figure}

The following was shown in \cite[Corollary 11.4]{IU} for consistent dimer quivers; here we give a new and independent proof using columns and pillars.

\begin{Theorem} \label{AKZ}
Suppose (i) $A$ is cancellative, or (ii) $\hat{\mathcal{C}}_i^u \not = \emptyset$ for each $u \in \mathbb{Z}^2$ and $i \in Q_0$.
Then $Q_1^{\mathcal{S}} = \emptyset$, that is, each arrow annihilates a simple $A$-module of dimension $1^{Q_0}$.
\end{Theorem}

\begin{proof}
Recall that assumption (i) implies assumption (ii), by Proposition \ref{here}.
So suppose (ii) holds, and assume to the contrary that there is an arrow $\delta$ in $Q_1^{\mathcal{S}}$.

By Lemma \ref{purple lemma}, there is a cycle $p \in \hat{\mathcal{C}}_{\operatorname{t}(\delta)}$ whose rightmost arrow subpath is not $\delta$ (modulo $I$).
Let $u \in \mathbb{Z}^2$ be such that $p \in \mathcal{C}^u$.
By assumption, there is a cycle $q$ in $\hat{\mathcal{C}}^u_{\operatorname{h}(\delta)}$.
By Lemma \ref{figure2lemma}, we may choose representatives $\tilde{p}$, $\tilde{q}$ of $p$, $q$ such that $\mathcal{R}_{\tilde{p},\tilde{q}}$ contains no vertices in its interior.
We thus have one of the three cases given in Figure \ref{AKZfig3}.

First suppose $\tilde{p}^+$ and $\tilde{q}^+$ do not intersect (that is, do not share a common vertex), as shown in case (i).
Then $\tilde{p}^+$ and $\tilde{q}^+$ bound a column.
By Lemma \ref{figure3lemma}, the brown arrows belong to a simple matching $x \in \mathcal{S}$.
In particular, $\delta$ is in $x$, contrary to assumption.

So suppose $\tilde{p}^+$ and $\tilde{q}^+$ intersect, as shown in cases (ii) and (iii).
Then $\tilde{p}^+$ and $\tilde{q}^+$ bound a union of pillars.
Again by Lemma \ref{figure3lemma}, the brown arrows belong to a simple matching $x \in \mathcal{S}$.
In particular, in case (ii) $\delta$ is in $x$, contrary to assumption.
Therefore case (iii) holds.
But then $\delta$ is a rightmost arrow subpath of $p$ (modulo $I$), contrary to our choice of $p$.
\end{proof}

In \cite[Theorem 1.1]{B3} we show that the converse of Theorem \ref{AKZ} also holds: $A$ is cancellative if and only if each arrow of $Q$ is contained in a simple matching.

\begin{figure}
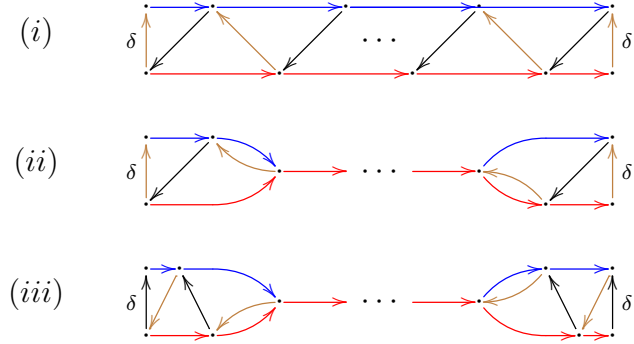

$$\begin{array}{ccc}
(i) & &
\xy 0;/r.3pc/:
(-24.5,-3.5)*{\cdot}="1";
(-10.5,-3.5)*{\cdot}="3";(3.5,-3.5)*{\cdot}="5";(24.5,-3.5)*{\cdot}="6";
(-24.5,3.5)*{\cdot}="7";(-17.5,3.5)*{\cdot}="8";
(-3.5,3.5)*{\cdot}="9";
(10.5,3.5)*{\cdot}="9'";(24.5,3.5)*{\cdot}="11";
(17.5,-3.5)*{\cdot}="5'";
(0,0)*{\cdots}="";
{\ar@[blue]"7";"8"};{\ar@[blue]"8";"9"};{\ar@[blue]"9";"11"};
{\ar@[red]"1";"3"};{\ar@[red]"3";"5"};{\ar@[red]"5'";"6"};
{\ar@[brown]_{\delta}"6";"11"};{\ar"11";"5'"};{\ar@[brown]"5'";"9'"};{\ar"9";"3"};{\ar@[brown]"3";"8"};{\ar"8";"1"};{\ar@[brown]^{\delta}"1";"7"};
{\ar@[blue]"9";"9'"};
{\ar"9'";"5"};{\ar@[red]"5";"5'"};
\endxy
\\
\\
(ii) &  &
\xy 0;/r.3pc/:
(-24.5,-3.5)*{\cdot}="1";
(-10.5,0)*{\cdot}="3";(17.5,-3.5)*{\cdot}="5";(24.5,-3.5)*{\cdot}="6";
(-24.5,3.5)*{\cdot}="7";(-17.5,3.5)*{\cdot}="8";(10.5,0)*{\cdot}="9";(24.5,3.5)*{\cdot}="11";
(-17.5,-3.5)*{}="2";(17.5,3.5)*{}="10";
(-3.5,0)*{}="14";(3.5,0)*{}="15";(0,0)*{\cdots}="";
{\ar@[red]@{->}"3";"14"};{\ar@[red]@{->}"15";"9"};
{\ar@[blue]"7";"8"};{\ar@[blue]@/^.3pc/"8";"3"};{\ar@[blue]@{-}@/^.3pc/"9";"10"};{\ar@[blue]"10";"11"};
{\ar@[red]@{-}"1";"2"};{\ar@[red]@/_.3pc/"2";"3"};{\ar@[red]@/_.3pc/"9";"5"};{\ar@[red]"5";"6"};
{\ar@[brown]_{\delta}"6";"11"};{\ar"11";"5"};{\ar@[brown]@/_.3pc/"5";"9"};
{\ar@[brown]@/^.3pc/"3";"8"};{\ar"8";"1"};{\ar@[brown]^{\delta}"1";"7"};
\endxy
\\
\\
(iii) & &
\xy 0;/r.3pc/:
(-24.5,-3.5)*{\cdot}="1";
(-10.5,0)*{\cdot}="3";(17.5,-3.5)*{}="5";(24.5,-3.5)*{\cdot}="6";
(-24.5,3.5)*{\cdot}="7";(-17.5,3.5)*{}="8";(10.5,0)*{\cdot}="9";(24.5,3.5)*{\cdot}="11";
(-17.5,-3.5)*{\cdot}="2";(17.5,3.5)*{\cdot}="10";
(-21,3.5)*{\cdot}="12";(21,-3.5)*{\cdot}="13";
(-3.5,0)*{}="14";(3.5,0)*{}="15";(0,0)*{\cdots}="";
{\ar@[red]@{->}"3";"14"};{\ar@[red]@{->}"15";"9"};
{\ar@[blue]"7";"12"};{\ar@[blue]@{-}"12";"8"};{\ar@[blue]@/^.3pc/"8";"3"};{\ar@[blue]@/^.3pc/"9";"10"};{\ar@[blue]"10";"11"};
{\ar@[red]"1";"2"};{\ar@[red]@/_.3pc/"2";"3"};{\ar@[red]@{-}@/_.3pc/"9";"5"};{\ar@[red]"5";"13"};{\ar@[red]"13";"6"};
{\ar_{\delta}"6";"11"};{\ar@[brown]"11";"13"};{\ar"13";"10"};{\ar@[brown]@/^.3pc/"10";"9"};
{\ar@[brown]@/_.3pc/"3";"2"};{\ar"2";"12"};{\ar@[brown]"12";"1"};{\ar^{\delta}"1";"7"};
\endxy
\end{array}$$
\caption{Setup for Theorem \ref{AKZ}.
The red and blue paths are the (lifts of the) cycles $p$ and $q$, respectively.
The red and blue arrows are paths in $Q^+$, and the black and brown arrows are arrows in $Q^+$.
In cases (i) and (ii), $\delta$ belongs to a simple matching, contrary to assumption.
In case (iii), $\delta$ is a rightmost arrow subpath of $p$ (modulo $I$), again contrary to assumption.}
\label{AKZfig3}
\end{figure}

\begin{Lemma} \label{easy injective}
Suppose $A$ is cancellative.
Let $p,q \in e_jAe_i$ be paths satisfying
$$\operatorname{t}(p^+) = \operatorname{t}(q^+) \ \ \text{ and } \ \ \operatorname{h}(p^+) = \operatorname{h}(q^+).$$
If $\overbar{p} = \overbar{q}$, then $p = q$.
\end{Lemma}

\begin{proof}
Since $A$ is cancellative, $Q$ has at least one simple matching by Lemma \ref{at least one}.
In particular, $\sigma \not = 1$.
Thus we may apply the proof of Lemma \ref{here2}.3, with $\bar{\tau}$ in place of $\bar{\eta}$.
\end{proof}

\begin{Lemma} \label{generated by}
Suppose $A$ is cancellative.
For each $i \in Q_0$, the corner ring $e_iAe_i$ is generated by $\sigma_i$ and $\hat{\mathcal{C}}_i$.
\end{Lemma}

\begin{proof}
Since $I$ is generated by binomials, $e_iAe_i$ is generated by $\mathcal{C}_i$.
It thus suffices to show that $\mathcal{C}_i$ is generated by $\sigma_i$ and $\hat{\mathcal{C}}_i$.

Let $u \in \mathbb{Z}^2$ and $p \in \mathcal{C}_i^u$.
If $u = 0$, then $p = \sigma_i^m$ for some $m \geq 0$ by Lemma \ref{here3'}.2.
So suppose $u \not = 0$.
Then there is a cycle $q$ in $\hat{\mathcal{C}}^u_i$ by Proposition \ref{here}.
In particular, $\overbar{p} = \overbar{q}\sigma^m$ for some $m \in \mathbb{Z}$ by Lemma \ref{longlist}.
Furthermore, $m \geq 0$ by Proposition \ref{circle 1}.1.
Therefore $p = q \sigma_i^m$ by Lemma \ref{easy injective}.
\end{proof}

It is well known that if $A$ is cancellative, then $A$ is a 3-Calabi-Yau algebra \cite{D,MR}.
In particular, the center $Z$ of $A$ is noetherian, and $A$ is a finitely generated $Z$-module.
In the following, we give independent proofs of these facts.

\begin{Proposition}  \label{i=j}
Suppose $A$ is cancellative, and let $i,j \in Q_0$.
Then
\begin{enumerate}
 \item $e_iAe_i = Ze_i \cong Z$.
 \item $\bar{\tau}\left(e_iAe_i \right) = \bar{\tau}\left(e_jAe_j\right)$.
 \item $A$ is a finitely generated $Z$-module, and $Z$ is a finitely generated $k$-algebra.
\end{enumerate}
\end{Proposition}

\begin{proof}
(1) For each $i \in Q_0$ and $u \in \mathbb{Z}^2 \setminus 0$, there exists a unique cycle $c_{ui} \in \hat{\mathcal{C}}_i^u$ (modulo $I$) by Propositions \ref{here} and \ref{circle 1}.4.
Thus the sum
$$\sum_{i \in Q_0} c_{ui} \in \bigoplus_{i \in Q_0} e_iAe_i$$
is in $Z$, by Lemma \ref{longlist3}.1.
Whence $e_iAe_i \subseteq Ze_i$ by Lemma \ref{generated by}.
Furthermore,
$$Ze_i = Ze_i^2 = e_iZe_i \subseteq e_iAe_i.$$
Therefore $Ze_i = e_iAe_i$.

We now claim that there is an algebra isomorphism $Z \cong Ze_i$ for each $i \in Q_0$.
Indeed, fix $i \in Q_0$ and suppose $z \in Z$ is nonzero.
Then there is some $j \in Q_0$ such that $ze_j \not = 0$.
Furthermore, since $Q$ is a dimer quiver, there is a path $p$ from $i$ to $j$.

Assume to the contrary that $ze_jp = 0$.
Thus, since $I$ is generated by binomials, it suffices to suppose $ze_j = c_1 - c_2$ where $c_1$ and $c_2$ are paths.
But since $A$ is cancellative, $ze_jp = 0$ implies $c_1 = c_2$.
Whence $ze_j = 0$, a contradiction.
Therefore $ze_j p \not = 0$.
Consequently,
$$p e_i z = pz = zp = ze_j p \not = 0.$$
Whence $ze_i \not = 0$.
Thus the algebra homomorphism $Z \to Ze_i$, $z \mapsto ze_i$, is injective, hence an isomorphism.
This proves our claim.

(2) Follows from Proposition \ref{circle 1}.2 and Lemma \ref{generated by}.

(3) $A$ is generated as a $Z$-module by all paths of length at most $|Q_0|$ by Claim (1) and \cite[second paragraph of proof of Theorem 2.11]{B6}.
Thus $A$ is a finitely generated $Z$-module.
Furthermore, $A$ is a finitely generated $k$-algebra since $Q$ is finite.
Therefore $Z$ is also a finitely generated $k$-algebra \cite[1.1.3]{McR}.
\end{proof}

\begin{Lemma} \label{finally!'}
Suppose $A$ is cancellative, and let $p \in \mathcal{C}^u$.
Then $u = 0$ if and only if $\overbar{p} = \sigma^m$ for some $m \geq 0$.
\end{Lemma}

\begin{proof}
($\Rightarrow$) Lemma \ref{here3'}.1.

($\Leftarrow$) Let $u \in \mathbb{Z}^2 \setminus 0$, and assume to the contrary that $p \in \mathcal{C}^u$ satisfies $\overbar{p} = \sigma^m$ for some $m \geq 0$.
Since $A$ is cancellative, there is a cycle $q \in \hat{\mathcal{C}}^u_{\operatorname{t}(p)}$ by Proposition \ref{here}.
Furthermore, $\sigma \nmid \overbar{q}$ by Proposition \ref{circle 1}.1.
Thus there is some $n \geq 0$ such that
$$\overbar{p} = \overbar{q}\sigma^n,$$
by Lemma \ref{here2}.2.
Whence $\sigma^m = \overbar{q} \sigma^n$.
Furthermore, $\sigma \not = 1$ by Lemma \ref{at least one}.
Therefore $m = n$ and
\begin{equation} \label{bar eta q = 1}
\overbar{q} = 1,
\end{equation}
since $B$ is a polynomial ring.
But each arrow in $Q$ is contained in a simple matching by Theorem \ref{AKZ}.
Therefore $\overbar{q} \not = 1$, contrary to (\ref{bar eta q = 1}).
\end{proof}

Recall that $B := k[\mathcal{S}]$ is the polynomial ring generated by the simple matchings $\mathcal{S}$.

\begin{Proposition} \label{injective prop}
Suppose $A$ is cancellative.
The homomorphisms
$$\tau: A \to M_{|Q_0|}(B) \ \ \ \text{ and } \ \ \ \eta: A \to M_{|Q_0|}\left(k[\mathcal{P}]\right)$$
are injective.
\end{Proposition}

\begin{proof}
(i) We fist claim that $\tau$ is injective on the vertex corner rings $e_iAe_i$, $i \in Q_0$.
Fix a vertex $i \in Q_0$ and let $p,q \in e_iAe_i$ be cycles satisfying $\overbar{p} = \overbar{q}$.
Let $r$ be a path such that $r^+$ is path from $\operatorname{h}(p^+)$ to $\operatorname{t}(p^+)$.
Then $rp \in \mathcal{C}^0$.
Thus there is some $m \geq 0$ such that
$$rp = \sigma_i^m,$$
by Lemma \ref{here3'}.2.
Whence
$$\overbar{rq} = \overbar{r}  \overbar{q} = \overbar{r}  \overbar{p} = \overbar{rp} = \overbar{\sigma_i^m} = \overbar{\sigma}_i^m = \sigma^m.$$
Thus $rq \in \mathcal{C}^0$ by Lemma \ref{finally!'}.
Hence $rq = \sigma_i^m = rp$ by Lemma \ref{easy injective}.
Therefore $p = q$ since $A$ is cancellative.

(ii) We now claim that $\tau$ is injective on paths.
Let $p,q \in e_jAe_i$ be paths satisfying $\overbar{p} = \overbar{q}$.
Let $r$ be a path from $j$ to $i$.
The two cycles $pr$ and $qr$ at $j$ then satisfy $\overbar{pr} = \overbar{qr}$.
Thus $pr = qr$ since $\tau$ is injective on the corner ring $e_jAe_j$ by Claim (i).
Therefore $p = q$ since $A$ is cancellative.

(iii) Since $A$ is generated by paths and $\tau$ is injective on paths, it follows that $\tau$ is injective.

(iv) Finally, we claim that $\eta$ is injective.
Let $g \in A$, and suppose $\eta(g) = 0$.
Then
\begin{equation} \label{tau g}
\tau(g) = \eta(g)|_{x = 1 \, : \, x \not \in \mathcal{S}} = 0.
\end{equation}
Therefore $g = 0$ by Claim (iii).
\end{proof}

\begin{Theorem} \label{arehom}
Cancellative dimer algebras are ghor algebras.
\end{Theorem}

\begin{proof}
Let $A = kQ/I$ be cancellative.
We want to show that $I = \ker \eta$.

Indeed,  $I \subseteq \ker \eta$ by Lemma \ref{tau'A'}.
So let $g \in kQ$, and suppose $\eta(g) = 0$.
Then $\tau(g) = 0$ by (\ref{tau g}).
Whence $g \in I$ by Proposition \ref{injective prop}.
Therefore $\ker \eta \subseteq I$.
\end{proof}

\section{Proof of main theorem} \label{Cancellative dimer algebras}

Throughout, let $A = kQ/I$ be a dimer algebra, and let $\psi: A \to A'$ be a cyclic contraction to a cancellative dimer algebra $A' = kQ'/I'$.
If $A$ is cancellative, then we may take $\psi$ to be the identity map.

Let $\tau: A' \to M_{|Q'_0|}(B)$ be the algebra homomorphism defined in Lemma \ref{tau'A'}, with $B$ the polynomial ring generated by the simple matchings of $A'$,
$$B = k[\mathcal{S}'].$$
To prove Theorem \ref{main}, we will use the composition of $\psi$ with $\tau$.
Specifically, let
$$\tau_{\psi}: A \to M_{|Q_0|}(B)$$
be the $k$-linear map defined for each $i,j \in Q_0$ and $p \in e_jAe_i$ by
\begin{equation*} \label{tilde tau define}
p \mapsto \bar{\tau}\psi(p) \cdot e_{ji}.
\end{equation*}
Denote by $\overbar{p} := \overbar{\tau}_{\psi}(p) := \bar{\tau}\psi(p)$ the single nonzero matrix entry of $\tau_{\psi}(p)$.

\begin{Lemma} \label{homomorphism on corner}
The map $\tau_{\psi}: A \to M_{|Q_0|}(B)$ is an algebra homomorphism.
\end{Lemma}

\begin{proof}
By Lemma \ref{tau'A'}, $\tau: A' \to M_{|Q'_0|}(B)$ is an algebra homomorphism.
Furthermore, $\psi$ is a $k$-linear map, and an algebra homomorphism when restricted to each vertex corner ring $e_iAe_i$.
\end{proof}

\begin{Lemma} \label{here4}
Let $p$ be a nontrivial cycle.
\begin{enumerate}
 \item If $p \in \mathcal{C}^0$, then $\overbar{p} = \sigma^m$ for some $m \geq 1$.
 \item If $p \in \mathcal{C} \setminus \hat{\mathcal{C}}$, then $\sigma \mid \overbar{p}$.
\end{enumerate}
\end{Lemma}

\begin{proof}
If $p^+$ is a cycle (resp.\ has a cyclic subpath) in $Q^+$, then $\psi(p)^+$ is a cycle (resp.\ has a cyclic subpath) in $Q'^+$.
Furthermore, $A'$ is cancellative.
Claims (1) and (2) therefore hold by Lemmas \ref{here3'}.1 and \ref{here3'}.3 respectively.
\end{proof}

\begin{Lemma} \label{coincident}
If $\psi(p) = \psi(q)$, then $p^+$ and $q^+$ have coincident tails and coincident heads.
\end{Lemma}

\begin{proof}
Suppose $\psi(p) = \psi(q)$.
Consider lifts $p^+$ and $q^+$ for which $\operatorname{t}(p^+) = \operatorname{t}(q^+)$.
Let $r^+$ be a path from $\operatorname{h}(p^+)$ to $\operatorname{h}(q^+)$.
Then $\psi(r^+)$ is a cycle in $Q'^+$ since $\psi(p) = \psi(q)$, by Lemma \ref{r+}.1.
Thus $r^+$ is also a cycle by the contrapositive of Lemma \ref{positive length cycle}.2.
Whence $\operatorname{h}(p^+) = \operatorname{h}(q^+)$.
\end{proof}

Denote by $\mathcal{P}$ and $\mathcal{P}'$ the set of perfect matchings of $Q$ and $Q'$ respectively.
Consider the algebra homomorphisms defined in (\ref{eta def}),
$$\eta: kQ \to M_{|Q_0|}\left( k[\mathcal{P}] \right) \ \ \ \text{ and } \ \ \ \eta': kQ' \to M_{|Q'_0|}\left( k[\mathcal{P}'] \right).$$
By Lemma \ref{cannot contract}, $\psi$ cannot contract a unit cycle to a vertex.
Thus, if $y$ is a perfect matching of $Q'$, then $\psi^{-1}(y)$ is a perfect matching of $Q$.
We may therefore view $k[\mathcal{P}']$ as a subalgebra of $k[\mathcal{P}]$ under the identification $y = \psi^{-1}(y)$ for each $y \in \mathcal{P}'$.
For $g \in e_jkQe_i$, we then have
\begin{equation} \label{eta' psi}
\bar{\eta}'(\psi(g)) = \bar{\eta}(g)|_{x = 1 \,: \, \psi(x) \not \in \mathcal{P}'}.
\end{equation}

Denote the $\bar{\eta}$- and $\bar{\eta}'$-images of the unit cycles in $Q$ and $Q'$ by
$$\sigma_{\mathcal{P}} := \prod_{x \in \mathcal{P}} x \ \ \ \text{ and } \ \ \ \sigma_{\mathcal{P}'} := \prod_{y \in \mathcal{P}'} y.$$

\begin{Proposition}
The $k$-linear map $\psi: kQ \to kQ'$ induces a $k$-linear map of ghor algebras $\psi: \Lambda \to \Lambda'$.
Furthermore, for each $i,j \in Q_0$, the restriction
\begin{equation} \label{psi: ej}
\psi: e_j\Lambda e_i \hookrightarrow e_{\psi(j)} \Lambda' e_{\psi(i)}
\end{equation}
is injective.
\end{Proposition}

\begin{proof}
(i) If $g \in kQ$ satisfies $\eta(g) = 0$, then $\eta'(\psi(g)) = 0$, by (\ref{eta' psi}).
Thus,
$$\psi(\ker \eta) \subseteq \ker \eta'.$$
Therefore the map $\psi: kQ \to kQ'$ induces a well-defined map $\psi: \Lambda \to \Lambda'$.

(ii) Fix $i,j \in Q_0$.
We claim that the map (\ref{psi: ej}) is injective.
Let $p,q \in e_jkQe_i$ be paths satisfying $\bar{\eta}'(\psi(p-q)) = 0$.
To prove the claim, it suffices to show that $\bar{\eta}(p - q) = 0$.

Since $\bar{\eta}'(\psi(p-q)) = 0$, we have $\psi(p) = \psi(q)$ by Proposition \ref{injective prop}.
In particular, $p^+$ and $q^+$ have coincident tails and coincident heads, by Lemma \ref{coincident}.
Let $r^+$ be a path in $Q^+$ from $\operatorname{h}(p^+)$ to $\operatorname{t}(p^+)$.
Then there is some $m, n \geq 0$ such that
\begin{equation} \label{eta p =}
\bar{\eta}(rp) = \sigma^m_{\mathcal{P}} \ \ \ \text{ and } \ \ \ \bar{\eta}(rq) = \sigma^n_{\mathcal{P}},
\end{equation}
by Lemma \ref{here4}.1.
It follows that
\begin{equation} \label{yippy}
\bar{\eta}(p) = \bar{\eta}(q) \sigma_{\mathcal{P}}^{m-n}.
\end{equation}
Therefore
$$\bar{\eta}'(\psi(q)) \stackrel{\textsc{(i)}}{=} \bar{\eta}'(\psi(p)) \stackrel{\textsc{(ii)}}{=} \bar{\eta}(p)|_{x = 1\, : \, \psi(x) \not \in \mathcal{P}'}
\stackrel{\textsc{(iii)}}{=} \bar{\eta}(q)\sigma_{\mathcal{P}}^{m-n}|_{x = 1\, : \, \psi(x) \not \in \mathcal{P}'} \stackrel{\textsc{(iv)}}{=} \bar{\eta}'(\psi(q))\sigma_{\mathcal{P}'}^{m-n},$$
where (\textsc{i}) holds by assumption; (\textsc{ii}) and (\textsc{iv}) hold by (\ref{eta' psi}); and (\textsc{iii}) holds by (\ref{yippy}).
But then $m = n$ since $B$ is an integral domain and $\sigma_{\mathcal{P}'} \not = 1$.
Thus $\bar{\eta}(p) = \bar{\eta}(q)$, by (\ref{eta p =}).
Therefore $\eta(p - q) = 0$.
\end{proof}

The following strengthens Lemmas \ref{here2}.3 and \ref{here2}.4 for dimer algebras that admit cyclic contractions (specifically, the head and tail of the lifts $p^+$ and $q^+$ are not required to coincide).

\begin{Lemma} \label{r in T'}
Let $p,q \in e_jAe_i$ be distinct paths.
The following are equivalent:
 \begin{enumerate}
   \item $\psi(p) = \psi(q)$.
   \item $p,q$ is a non-cancellative pair.
   \item $\overbar{\tau}_{\psi}(p) = \overbar{\tau}_{\psi}(q)$.
   \item $\bar{\eta}(p) = \bar{\eta}(q)$.
 \end{enumerate}
\end{Lemma}

\begin{proof}
(1) $\Rightarrow$ (2): Suppose $\psi(p) = \psi(q)$.
Then $p^+$ and $q^+$ have coincident tails and coincident heads, by Lemma \ref{coincident}.
Furthermore, $\mathcal{P}' \not = \emptyset$ since $A'$ is cancellative, by Lemma \ref{at least one}.
Whence $\mathcal{P} \not = \emptyset$, by Lemma \ref{cannot contract}.
Therefore $p,q$ is a non-cancellative pair, by Lemma \ref{here2}.4 (with $\bar{\tau}_{\psi}$ in place of $\bar{\eta}$).

(2) $\Rightarrow$ (3),(4): Holds by Lemma \ref{here2}.3 (with $\bar{\tau}_{\psi}$ in place of $\bar{\tau}$).

(3) $\Rightarrow$ (1): Holds since $\bar{\tau}: e_{\psi(j)}A'e_{\psi(i)} \to B$ is injective, by Proposition \ref{injective prop}.

(4) $\Rightarrow$ (3): If $\bar{\eta}(p) = \bar{\eta}(q)$, then
$$\bar{\tau}_{\psi}(p) = \bar{\eta}(p)|_{x = 1 \, : \, \psi(x) \not \in \mathcal{S}'} = \bar{\eta}(q)|_{x = 1 \, : \, \psi(x) \not \in \mathcal{S}'} = \bar{\tau}_{\psi}(q).$$
\end{proof}

The following theorem establishes the relationship between ghor algebras and dimer algebras on a torus.

\begin{Theorem} \label{first main}
There are algebra isomorphisms
$$\begin{array}{rcl}
\Lambda := kQ/\ker \eta & \stackrel{\textsc{(i)}}{=} & kQ/\ker \tau_{\psi} \\
& \stackrel{\textsc{(ii)}}{\cong} & A/\ker \tau_{\psi} \\
& \stackrel{\textsc{(iii)}}{=} & A/\left\langle p - q \ | \ p,q \text{ is a non-cancellative pair} \right\rangle.
\end{array}$$
\end{Theorem}

\begin{proof}
(\textsc{i}) and (\textsc{ii}) hold since
$$I  \stackrel{(a)}{\subseteq} \ker \eta \stackrel{(b)}{=} \ker \tau_{\psi}.$$
Indeed, ($a$) holds by Lemma \ref{tau'A'}, and ($b$) holds from the equivalence (3) $\Leftrightarrow$ (4) in Lemma \ref{r in T'}.
(\textsc{iii}) holds from the equivalence (2) $\Leftrightarrow$ (3) in Lemma \ref{r in T'}.
\end{proof}

\begin{Corollary} \label{injective cor}
The map $\tau_{\psi}: A \to M_{|Q_0|}(B)$ induces an injective algebra homomorphism on the ghor algebra $\Lambda$,
\begin{equation*} \label{tau psi lambda}
\tau_{\psi}: \Lambda \to M_{|Q_0|}(B).
\end{equation*}
\end{Corollary}

\begin{Remark} \rm{
The ideal
$$\left\langle p - q \ | \ p,q \text{ is a non-cancellative pair} \right\rangle \subset A$$
is contained in the kernel of $\psi$, but not conversely.
Indeed, if $\psi$ contracts an arrow $\delta$, then $\delta - e_{\operatorname{t}(\delta)}$ is in the kernel of $\psi$, but $\delta$ and $e_{\operatorname{t}(\delta)}$ do not form a non-cancellative pair.
}\end{Remark}

Since $A'$ is cancellative, $A'$ is equal to its ghor algebra $\Lambda'$ by Theorem \ref{arehom} (or Theorem \ref{first main}).

\begin{Theorem} \label{impression prop}
Let $\psi: \Lambda \to \Lambda'$ be a cyclic contraction.
\begin{enumerate}
 \item The algebra homomorphisms
$$\tau_{\psi}: \Lambda \to M_{|Q_0|}(B) \ \ \ \text{ and } \ \ \ \tau: \Lambda' \to M_{|Q'_0|}(B)$$
are impressions of $\Lambda$ and $\Lambda'$.
 \item Suppose that either $k$ is uncountable, or $Q$ is cancellative.
Then $\tau_{\psi}$ classifies all simple $\Lambda$-module isoclasses of maximal $k$-dimension: for each such module $V$, there is some $\mathfrak{b} \in \operatorname{Max}B$ such that
 $$V \cong (B/\mathfrak{b})^{|Q_0|},$$
  where $av := \tau_{\psi}(a)v$ for each $a \in \Lambda$, $v \in (B/\mathfrak{b})^{|Q_0|}$.
 \item The centers of $\Lambda$ and $\Lambda'$ are given by the intersection and union of the vertex corner rings of $\Lambda$,
\begin{equation} \label{Z cong R S}
Z(\Lambda) \cong k\left[ \cap_{i \in Q_0} \bar{\tau}_{\psi}\left( e_i \Lambda e_i \right) \right] \subseteq k\left[ \cup_{i \in Q_0} \bar{\tau}_{\psi}\left( e_i\Lambda e_i \right) \right] \cong Z(\Lambda').
\end{equation}
\end{enumerate}
\end{Theorem}

\begin{proof}
(1) We first show that $\tau_{\psi}$ and $\tau$ are impressions of $\Lambda$ and $\Lambda'$.

(1.i) $\tau_{\psi}: \Lambda \to M_{|Q_0|}(B)$ is injective by Corollary \ref{injective cor}.

(1.ii) For each maximal ideal
$$\mathfrak{b} \in \mathcal{Z}_B\left( \sigma \right)^c \subset \operatorname{Max}B = \mathbb{A}_k^{\left|\mathcal{S}\right|},$$
the composition
\begin{equation} \label{composition lambda}
\Lambda \stackrel{\tau_{\psi}}{\longrightarrow} M_{|Q_0|}(B) \stackrel{1}{\longrightarrow}  M_{|Q_0|}\left(B/\mathfrak{b} \right) \cong M_{|Q_0|}(k)
\end{equation}
is a simple representation of $\Lambda$.
Indeed, when viewed as a vector space diagram on $Q$ of dimension vector $1^{Q_0}$, each arrow is represented by a nonzero scalar.
Thus, since $Q$ has a cycle containing each vertex, the representation is simple.
It follows that the composition (\ref{composition lambda}) is surjective.

(1.iii) Set $R := \tau_{\psi}(Z(\Lambda))$; Claims (1.i) and (1.ii), together with  (\ref{Ziso}), imply that each element of $R$ is product of a polynomial in $B$ and the identity matrix $1_{|Q_0|} \in M_{|Q_0|}(B)$.
For brevity, we will omit $1_{|Q_0|}$ in our expressions.

We claim that the morphism
$$\operatorname{Max}B \rightarrow \operatorname{Max}R, \ \ \ \mathfrak{b} \mapsto \mathfrak{b} \cap R,$$
is surjective.
Indeed, for any $\mathfrak{m} \in \operatorname{Max}R$, $B\mathfrak{m}$ is a proper ideal of $B$.
Thus there is a maximal ideal $\mathfrak{b} \in \operatorname{Max}B$ containing $B\mathfrak{m}$ since $B$ is noetherian.
Furthermore, since $B$ is a finitely generated $k$-algebra and $k$ is algebraically closed, the intersection $\mathfrak{b} \cap R =: \mathfrak{m}'$ is a maximal ideal of $R$.
Whence
$$\mathfrak{m} \subseteq B\mathfrak{m} \cap R \subseteq \mathfrak{b} \cap R = \mathfrak{m}'.$$
But $\mathfrak{m}$ and $\mathfrak{m}'$ are both maximal ideals of $R$.
Thus $\mathfrak{m} = \mathfrak{m}'$.
Therefore $\mathfrak{b} \cap R = \mathfrak{m}$, proving our claim.

(1.iv) By Claims (1.i), (1.ii), and (1.iii), $\tau_{\psi}$ is an impression of $\Lambda$.
It follows that $\tau$ itself is an impression of $\Lambda'$ by letting $\psi: A = A' \to A'$ be the trivial cyclic contraction given by the identity map.

(2.i) If $Q$ is cancellative, then $\tau$ classifies all simple $\Lambda$-module isoclasses of maximal $k$-dimension, by (\ref{Vcong}) and Proposition \ref{i=j}.3.

(2.ii) Now suppose $k$ is uncountable.
Using Claim (2.i), it was shown in \cite[Proposition 3.10 and Theorem 3.11]{B2} that, irrespective of whether $Q$ is cancellative or non-cancellative, $\tau_{\psi}$ classifies all simple $\Lambda$-modules (and $A$-modules) of dimension vector $1^{Q_0}$, up to isomorphism.

We claim that the dimension vector of the simple $\Lambda$-modules of maximal $k$-dimension is $1^{Q_0}$.
Let $V$ be a simple $\Lambda$-module.
Then for each $i \in Q_0$, $e_iV$ is a simple $e_i\Lambda e_i$-module.
But $e_i \Lambda e_i$ is a commutative countably generated $k$-algebra, and $k$ is an uncountable algebraically closed field.
Thus, $e_iV = 0$ or $e_iV \cong k$.
Whence $\operatorname{dim}_k e_iV \leq 1$.
Furthermore, there is a representation of $\Lambda$ where each arrow is represented by $1 \in k$ since we may set each $x$ equal to $1$; this representation is simple of dimension $1^{Q_0}$ since $Q$ has a cycle containing each vertex.
In particular, there is a simple $\Lambda$-module of dimension $1^{Q_0}$, proving our claim.

(3) We claim that
\begin{multline*}
Z(\Lambda) \stackrel{(\textsc{i})}{\cong}
k\left[ \cap_{i \in Q_0} \bar{\tau}_{\psi}\left( e_i \Lambda e_i \right) \right]
\subseteq k\left[ \cup_{i \in Q_0} \bar{\tau}_{\psi}\left( e_i\Lambda e_i \right) \right]
\\ \stackrel{(\textsc{ii})}{=}
k\left[ \cup_{i \in Q'_0} \bar{\tau}\left( e_i\Lambda' e_i \right) \right]
\stackrel{(\textsc{iii})}{=}
k\left[ \cap_{i \in Q'_0} \bar{\tau}\left( e_i\Lambda' e_i \right) \right]
\stackrel{(\textsc{iv})}{\cong} Z(\Lambda').
\end{multline*}
(\textsc{i}) and (\textsc{iv}) hold by Claims (1.i) and (1.ii), together with (\ref{Ziso}).
(\textsc{ii}) holds since the contraction $\psi$ is cyclic, and since $A' = \Lambda'$.
Finally, (\textsc{iii}) holds by Proposition \ref{i=j}.2.
Therefore (\ref{Z cong R S}) holds.
\end{proof}

\begin{Remark} \rm{
In the case of cancellative dimer algebras, the labeling of arrows given in (\ref{taua}) agrees with the labeling of arrows in the toric construction of \cite[Proposition 5.3]{CQ}.
We note, however, that impressions are defined more generally for non-toric algebras and have different implications than the toric construction of \cite{CQ}.
}\end{Remark}

\begin{Corollary} \label{cr}
A dimer algebra $A$ is cancellative if and only if it admits an impression $\tau: A \to M_{|Q_0|}(B)$, where $B$ is an integral domain and $\tau(e_i) = e_{ii}$ for each $i \in Q_0$.
\end{Corollary}

\begin{proof}
Suppose $A$ admits an impression $\tau: A \to M_{|Q_0|}(B)$, where $B$ is an integral domain and $\tau(e_i) = e_{ii}$ for each $i \in Q_0$.
Consider paths $p,q,r$ satisfying $pr = qr \not = 0$.
Then
$$\overbar{p}  \overbar{r} = \overbar{pr} = \overbar{qr} = \overbar{q}  \overbar{r}.$$
Thus $\overbar{p} = \overbar{q}$ since $B$ is an integral domain.
Whence
$$\tau(p) = \overbar{p}e_{\operatorname{h}(p),\operatorname{h}(r)} = \overbar{q}e_{\operatorname{h}(p),\operatorname{h}(r)} = \tau(q).$$
Therefore $p = q$ by the injectivity of $\tau$.

The converse holds by Theorem \ref{impression prop}.
\end{proof}

Recall that an algebra $A$ is prime if for all $a,b \in A$, $aAb = 0$ implies $a = 0$ or $b = 0$, that is, the zero ideal is a prime ideal.

\begin{Corollary} \label{cancellative prime}
Ghor algebras are prime.
\end{Corollary}

\begin{proof}
We claim that for nonzero elements $p,q \in A$, we have $qAp \not = 0$.
It suffices to suppose that
$$p \in e_{j}Ae_{i} \ \ \text{ and } \ \ q \in e_{\ell}Ae_{k}.$$
Since $Q$ is a dimer quiver, there is a path $r$ from $j$ to $k$.
Furthermore, the polynomials $\overbar{p}$, $\overbar{q}$, and $\overbar{r}$ are nonzero since $\tau$ is injective, by Theorem \ref{impression prop}.
Thus the product $\overbar{qrp} = \overbar{q} \overbar{r} \overbar{p} \in B$ is nonzero since $B$ is an integral domain.
Therefore $qrp$ is nonzero.
\end{proof}

\begin{Example}  \label{four examples} \rm{
A dimer algebra $A = kQ/I$ is \textit{square} if the underlying graph of its cover $Q^+$ is a square grid graph with vertex set $\mathbb{Z} \times \mathbb{Z}$, and with at most one diagonal edge in each unit square.
Square dimer algebras exhibit a remarkable interaction between their torus embedding and their representation theory.
Specifically, such an algebra $A$ admits an impression $\tau$ where for each arrow $a \in Q^+_1$, $\bar{\tau}(a)$ is the monomial corresponding to the orientation of $a$ given in Figure \ref{labels} \cite[Theorem 3.7]{B7}.
In this case, $B = k[x,y,z,w]$ is the polynomial ring in four variables.
If $Q$ only possesses three arrow orientations, say up, left, and right-down, then we may label the respective arrows by $x$, $y$, and $z$, and obtain an impression with $B = k[x,y,z]$.
In either case, $A$ is cancellative by Corollary \ref{cr}.

Consider the cyclic contractions $\psi: A \to A'$ given in Figure \ref{deformation figure}.
In each example, $A'$ is a square dimer algebra (with the 2-cycles removed from $Q'$).
By Theorem \ref{impression prop}, the cycle algebra $S$ and center $R$ of the ghor algebra $\Lambda = kQ/\ker \eta$ are respectively
$$\begin{array}{cclcl}
\text{(i):} & \ \ & S = k\left[ xz, xw, yz, yw \right] & & R = k + \sigma S \\
\text{(ii):} & \ \ & S = k\left[ xz, xw, yz, yw \right] & & R = k + (x^2zw, y^2zw, \sigma)S \\
\text{(iii):} & \ \ & S = k\left[ xz, yz, xw, yw \right] & & R = k + (xz, yz)S \\
\text{(iv):} & \ \ & S = k\left[ xz, yw, x^2w^2, y^2z^2 \right] & & R = k + (yw, x^2w^2, y^2z^2)S
\end{array}$$
}\end{Example}

\begin{figure}
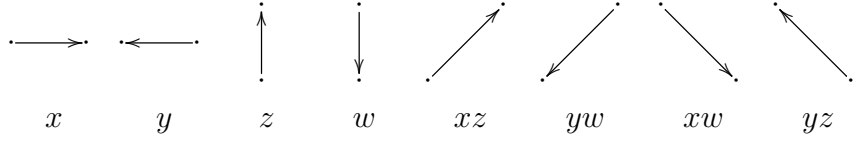

$$\begin{array}{cccccccc}
\xy
(-5,0)*{\cdot}="1";(5,0)*{\cdot}="2";{\ar@{->}"1";"2"};
\endxy
&
\xy
(-5,0)*{\cdot}="1";(5,0)*{\cdot}="2";{\ar@{<-}"1";"2"};
\endxy
&
\xy
(0,-5)*{\cdot}="1";(0,5)*{\cdot}="2";{\ar@{->}"1";"2"};
\endxy
&
\xy
(0,5)*{\cdot}="1";(0,-5)*{\cdot}="2";{\ar@{->}"1";"2"};
\endxy
&
\xy
(-5,-5)*{\cdot}="1";(5,5)*{\cdot}="2";{\ar@{->}"1";"2"};
\endxy
&
\xy
(-5,-5)*{\cdot}="1";(5,5)*{\cdot}="2";{\ar@{<-}"1";"2"};
\endxy
&
\xy
(-5,5)*{\cdot}="1";(5,-5)*{\cdot}="2";{\ar@{->}"1";"2"};
\endxy
&
\xy
(-5,5)*{\cdot}="1";(5,-5)*{\cdot}="2";{\ar@{<-}"1";"2"};
\endxy
\\
\ \ \ x \ \ \ & \ \ \ y \ \ \ & \ \ \ z \ \ \ & \ \ \ w \ \ \ & \ \ \ xz \ \ \ & \ \ \ yw \ \ \ & \ \ \ xw \ \ \ & \ \ \ yz \ \ \
\end{array}$$
\caption{A labeling of arrows in the covering quiver of a square dimer algebra that specifies an impression.}
\label{labels}
\end{figure}

\begin{Proposition} \label{finally!}
Suppose $A$ is cancellative, and let $p \in \mathcal{C}$ be a nontrivial cycle.
Then $p \in \hat{\mathcal{C}}$ if and only if $\sigma \nmid \overbar{p}$.
In particular, $\tau(Z) \subseteq B$ is generated over $k$ by $\sigma$ and a set of monomials in $B$ not divisible by $\sigma$.
\end{Proposition}

\begin{proof}
($\Rightarrow$) Proposition \ref{circle 1}.1.

($\Leftarrow$) Lemma \ref{here3'}.3.
\end{proof}

\appendix
\section{A brief account of Higgsing with quivers} \label{Higgsing}

\begin{center}\textit{Quiver gauge theories}\end{center}

According to string theory, our universe is 10 dimensional.\footnote{Thanks to Francesco Benini, Mike Douglas, Peng Gao, Mauricio Romo, and James Sparks for discussions on the physics of non-cancellative dimers.}\textsuperscript{,}\footnote{More correctly, weakly coupled superstring theory requires 10 dimensions.}
In many string theories our universe has a product structure $M \times Y$, where $M$ is our usual 4-dimensional spacetime and $Y$ is a 6-dimensional compact Calabi-Yau variety.

Let us consider a special class of gauge theories called `quiver gauge theories', which can often be realized in string theory.\footnote{Here we are considering theories with $\mathcal{N} = 1$ supersymmetry.}
The input for such a theory is a quiver $Q$, a superpotential $W$, a dimension vector $d \in \mathbb{N}^{Q_0}$, and a stability parameter $\theta \in \mathbb{R}^{Q_0}$.

Let $I$ be the ideal in $\mathbb{C}Q$ generated by the partial derivatives of $W$ with respect to the arrows in $Q$.
These relations (called `F-term relations') are classical equations of motion from a supersymmetric Lagrangian with superpotential $W$.\footnote{More correctly, the F-term relations plus the D-term relations imply the equations of motion.}
Denote by $A$ the quiver algebra $\mathbb{C}Q/I$.

According to these theories, the space $X$ of $\theta$-stable representation isoclasses of dimension $d$ is an affine chart on the compact Calabi-Yau variety $Y$.
The `gauge group' of the theory is the isomorphism group (i.e., change of basis) for representations of $A$.

Physicists view the elements of $A$ as fields on $X$.
More precisely, $A$ may be viewed as a noncommutative ring of functions on $X$, where the evaluation of a function $f \in A$ at a point $p \in X$ (i.e., representation $p$) is the matrix $f(p) := p(f)$ (up to isomorphism).

\begin{center}
\textit{Vacuum expectation values}
\end{center}

Given a path $f \in A$ and a representation $p \in X$, denote by $f\left(\overbar{p}\right)$ the matrix representing $f$ in the vector space diagram on $Q$ associated to $p$.

A field $f \in A$ is `gauge-invariant' if $f(p) = f(p')$ whenever $p$ and $p'$ are isomorphic representations (i.e., they differ by a `gauge transformation').
If $f$ is a path, then $f$ will necessarily be a cycle in $Q$.

The `vacuum expectation value' of a field is its expected (average) energy in the vacuum (similar to rest mass), and is abbreviated `vev'.
In our case, the vev of a path $f \in A$ at a point $p \in X$ is the matrix $f\left(\overbar{p}\right)$, which is just the expected energy of $f$ in $M \times \{ p \}$.

\begin{center}
\textit{Higgsing}
\end{center}

Spontaneous symmetry breaking is a process where the symmetry of a physical system decreases, and a new property (typically mass) emerges.

For example, suppose a magnet is heated to a high temperature.
Then all of its molecules, which are each themselves tiny magnets, jostle and wiggle about randomly.
In this heated state the material has rotational symmetry and no net magnet field.
However, as the material cools, one molecule happens to settle down first.
As the neighboring molecules settle down, they align themselves with the first molecule, until all the molecules settle down in alignment with the first.\footnote{More precisely, there are domains of magnetization.}
The orientation of the first settled molecule then determines the direction of magnetization for the whole material, and the material no longer has rotational symmetry.
One says that the rotational symmetry of the heated magnet was spontaneously broken as it cooled, and a global magnetic field emerged.\footnote{This is an example of `global' symmetry breaking, meaning the symmetry is physically observable.}

Higgsing is a way of using spontaneous symmetry breaking to turn a quantum field theory with a massless field and more symmetry into a theory with a massive field and less symmetry.
Here mass (vev's) takes the place of magnetization, gauge symmetry (or the rank of the gauge group) takes the place of rotational symmetry, and energy scale (RG flow) takes the place of temperature.

The recent discovery of the Higgs boson at the Large Hadron Collider is another example of Higgsing.\footnote{This is an example `gauge' symmetry breaking, meaning the symmetry is not an actual observable symmetry of a physical system, but only an artifact of the math used to describe it (like a choice of basis for the matrix of a linear transformation).}

\begin{center}
\textit{Higgsing in quiver gauge theories}
\end{center}

We now give our main example.
Suppose an arrow $a$ in a quiver gauge theory with dimension $1^{Q_0}$ is contracted to a vertex $e$.
We make two observations:
\begin{enumerate}
  \item the rank of the gauge group drops by one since the head and tail of $a$ become identified as the single vertex $e$;
  \item $a$ has zero vev at any representation where $a$ is represented by zero, while $e$ can never have zero vev since it is a vertex, and $X$ only consists of representation isoclasses with dimension $1^{Q_0}$.
\end{enumerate}
We therefore see that contracting an arrow to a vertex is a form of Higgsing in quiver gauge theories with dimension $1^{Q_0}$.\footnote{This is another example of gauge symmetry breaking.}

In the context of a 4-dimensional $\mathcal{N} = 1$ quiver gauge theory with quiver $Q$, the Higgsing we consider in this article is related to RG flow.
We start with a non-superconformal (strongly coupled) quiver theory $Q$ which admits a low energy effective description, give nonzero vev's to a set of bifundamental fields $Q_1^*$, and obtain a new theory $Q'$ that lies at a superconformal fixed point.

\begin{center}\textit{The mesonic chiral ring and the cycle algebra}\end{center}

The cycle algebra $S$, introduced in \cite{B2}, is similar to the mesonic chiral ring in the corresponding quiver gauge theory.
In such a theory, the mesonic operators, which are the gauge invariant operators, are generated by the cycles in the quiver.
If the gauge group is abelian, then the dimension vector is $1^{Q_0}$.
In the case of a dimer theory with abelian gauge group, two disjoint cycles may share the same $\bar{\tau}\psi$-image, but take different values on a point of the vacuum moduli space.
These cycles would then be distinct elements in the mesonic chiral ring, although they would be identified in the cycle algebra $S$; see \cite[Remark 3.17]{B2}.

\ \\
\textbf{Acknowledgments.}
The author thanks an anonymous referee for their careful reading and helpful comments. 
The author was supported by the Austrian Science Fund (FWF) grants P34854 and P30549-N26.
Part of this article is also based on work supported by the Simons Foundation and the Heilbronn Institute for Mathematical Research while the author held postdoctoral positions at the Simons Center for Geometry and Physics at Stony Brook University and the University of Bristol.
%
%
%
%

\bibliographystyle{hep}
\def\cprime{$'$} \def\cprime{$'$}

\end{document}